\documentclass[12pt,reqno]{amsart}

\usepackage{amsmath,amssymb,amsthm}
\usepackage[margin=1.2in]{geometry}
\usepackage{enumitem,hyperref}
\usepackage{amsrefs}
\newtheorem{theorem}{Theorem}[section]
\newtheorem{proposition}[theorem]{Proposition}
\newtheorem{lemma}[theorem]{Lemma}
\newtheorem{corollary}[theorem]{Corollary}
\newtheorem{conjecture}[theorem]{Conjecture}
\theoremstyle{definition}
\newtheorem{definition}[theorem]{Definition}
\newtheorem{convention}[theorem]{Convention}
\theoremstyle{remark}
\newtheorem{remark}[theorem]{Remark}

\newcommand{\CC}{\mathbb{C}}
\newcommand{\ZZ}{\mathbb{Z}}
\newcommand{\RR}{\mathbb{R}}
\newcommand{\PP}{\mathbb{P}}
\newcommand{\Hom}{\operatorname{Hom}}
\newcommand{\Ext}{\operatorname{Ext}}
\newcommand{\RHom}{\operatorname{RHom}}
\newcommand{\Stab}{\operatorname{Stab}}
\newcommand{\HH}{\operatorname{HH}}
\newcommand{\Db}{\mathrm{D}^{\mathrm{b}}}
\newcommand{\AX}{\mathcal{A}_X}
\newcommand{\Mukai}{\widetilde{H}(\AX,\ZZ)}
\newcommand{\hb}{\hbar}
\newcommand{\ie}{i.e.}

\newcommand{\PI}{P$_{\mathrm{I}}$}
\providecommand{\QQ}{\mathbb{Q}}

\begin{document}

\title[K3 Atoms and the BPS determinant line]{K3 atoms of the cubic fourfold\\ and the BPS structure of the Painlev\'e~I determinant line}
\author{Mark Raugas}
\address{Pacific Northwest National Laboratory, 1100 Dexter Ave N,
Seattle, WA 98109}
\email{raugas@pnnl.gov}
\subjclass[2020]{Primary 14F08; Secondary 14N35, 34M40, 34M55, 34M60, 58J52}
\keywords{Bridgeland stability, semiorthogonal decomposition, Hodge
atom, Kuznetsov component, noncommutative minimal model program, BPS
structure, Painlev\'e I, cubic oscillator, exact WKB, spectral
determinant, tau function, Stokes phenomenon}

\begin{abstract}
The compatibility of the semiorthogonal decomposition of a cubic
fourfold with Bridgeland stability beyond a generic point is the
missing ingredient in the conjectured dynamical protection of the K3
atom $\AX$. We analyze it in the exactly solvable models of the
noncommutative minimal model program. In the uncoupled Fano models
($\PP^1$; $\PP^1 \times \PP^1$ at resonance) the quantum cohomology
path exits the geometric chamber at an explicit finite time, enters
the selection region of the gluing, and never leaves; perturbing the
resonance shows the $\varepsilon$-crossover is
not a wall. The $A_2$ quiver provides a coupled counterpart.  We
examine the deformed cubic oscillator, Painlev\'e~I, and
formulate a tau-JLO dictionary as a determinant identification and establish
its perturbative layer, including computing
closed-form quantum periods through $Z_4$,
all-orders flatness with exact curvature the tau-divisor current,
identifying it with the zeta determinant's zero divisor along a
Painlev\'e~I trajectory. We close by conjecturing that the
determinant's nonperturbative jumps are automorphisms of
the underlying Joyce structure.
\end{abstract}

\maketitle

\section{Introduction}\label{sec:intro}

The rationality problem for the cubic fourfold $X \subset \PP^5$ was
recently resolved for the very general member by Katzarkov,
Kontsevich, Pantev, and Yu \cite{KKPY}, who extract from the Stokes
data of the quantum connection a birational invariant, the $K3$
Hodge atom, supported on the rank-24 zero-eigenvalue block of
quantum multiplication by $c_1(X)$. In \cite{R} we proposed a physical
reading of this invariant: the semiorthogonal decomposition
\[
  \Db(X) = \bigl\langle\, \AX,\; \mathcal{O}_X,\; \mathcal{O}_X(1),\;
  \mathcal{O}_X(2)\,\bigr\rangle
\]
acts as a dynamical selection rule, the full acyclicity of
$\RHom(E, A)$ for exceptional $E$ and $A \in \AX$ forbidding not
merely a net BPS index but any stable tunneling trajectory between the
vacua of birational surgery and the $K3$ sector, so that the atom's
spectrum constitutes a protected quantum phase
\cite[Conjecture~8.1]{R}.

However,
the Kontsevich--Soibelman formula \cite{KS} guarantees
invariance of the ordered product of Stokes automorphisms, but does
not by itself prevent wall-crossing from creating binding BPS states
at special loci. So, to prove the above conjecture,
one would need to demonstrate the compatibility of the
semiorthogonal
decomposition with all stability conditions in a connected component
of $\Stab(\Db(X))$.

Two developments advance this question:

\begin{enumerate}
\item The noncommutative minimal model program of Halpern-Leistner \cite{HL},
with the quasi-convergence formalism of
Halpern-Leistner--Jiang--Robotis \cite{HLJR} and the augmented
stability space of Halpern-Leistner--Robotis \cite{RobAug}, provides
the mechanism by which the quantum differential equation selects paths
in $\Stab$ whose limits encode semiorthogonal decompositions.  That
program has been carried out for blowup surfaces by Karube \cite{Kar}
and extended to Grassmannians, quadrics,
and cubic threefolds and fourfolds by Karube--Robotis--Zuliani
\cite{KRZ}.
\item The stability theory of the Kuznetsov component
is itself well-developed: Bayer--Lahoz--Macr\`i--Stellari \cite{BLMS}
construct stability conditions on $\AX$ with support controlled by the
$A_2$ lattice.  Wall-crossing inside $\AX$ acts birationally on
hyperk\"ahler moduli \cite{BM14} while counting invariants are
stability-independent by the CY2 theorem of Toda \cite{Toda}.
\end{enumerate}

In this work,
we consider the locus of glued stability conditions in the sense of
Collins--Polishchuk \cite{CP}, whose phase windows separate, and
examine whether the quantum cohomology path enters the selection
region of the semiorthogonal decomposition \emph{at finite time} or
merely converges to that chamber's boundary in a
limit. Quasi-convergence is an asymptotic condition on
Harder--Narasimhan filtrations and is compatible with either behavior;
but the protection statement of \cite{R} is a statement about the
stable spectrum at interior points, and it holds on the selection
region, not on its closure. Equally undetermined by the general theory
is the fate of a degenerate spectral block along the path: whether the
flow can resolve, deform, or couple to the interior of an atom.

The purpose of this paper is to settle these questions completely in
the smallest models where they arise, and to extract from the
solutions two structural principles, provide
a terminality proposition for the
cubic fourfold, and map out an analytic continuation of the program.

\subsection*{The models}
For $\PP^1$ we compute the quantum path explicitly
(\S\ref{sec:P1}). Along any ray avoiding the finitely many Stokes
directions (due to quasi-convergence, phase slopes must separate) the
lifted phase gap between the exceptional objects grows linearly, and
the path crosses a single wall at an explicit finite time $t_0$.
The wall is the destabilization of the binding objects: the
torsion sheaves, whose classes are the unique classes mixing the two
sectors, become strictly semistable with Jordan--H\"older factors the
two exceptional objects. Past $t_0$ the path lies in the selection
region, which we verify is the entire unbounded region beyond the
wall: no new stable object ever forms. Dynamical protection is thus
\emph{realized along the flow}: the binding BPS states exist at
small $t$ and decouple at $t_0$.

For $\PP^1\times\PP^1$ at the resonance $q_1 = q_2$
(\S\ref{sec:product}), the factor-swap symmetry of the mirror forces
the central charges of the two eigenvalue-zero objects to coincide:
the block is exactly phase-locked for all time,
with central charge vanishing on the difference class
$\delta = [\mathcal{O}(1,0)] - [\mathcal{O}(0,1)]$. In particular, no ray can
resolve the block's interior. The path enters the three-block
selection region at finite time through walls of multiplicity two
and the two curve-class binding channels end simultaneously
(a degeneracy inherited from the resonance).  
The block persists to the limit
as a polynomially massless factor against exponentially separated
exceptional sectors. This is the minimal model of an atom the flow
cannot probe.

\subsection*{Two clocks: Phase and Mass}
Perturbing the resonance, $q_1 = qe^{i\varepsilon}$
(\S\ref{sec:epsilon}), splits the block's eigenvalues by
$\Delta\lambda = 2im\varepsilon + O(\varepsilon^2)$.
The internal separation is governed
by two parameters: a phase clock at rate
$\propto \varepsilon\cos\theta$ and a mass clock at rate
$\propto \varepsilon\sin\theta$, which are complementary projections of the
splitting onto the chosen ray.  The resulting crossover, occurring at
$t_1 \sim \pi/(2m\varepsilon|\cos\theta|)$, is not a wall.
Because the block's factors are orthogonal, there is no
gluing inequality and no binding class: the stable spectrum is
constant past the coarse entry time, independent of $\varepsilon$.

What changes at $t_1$ is the multi-scale type of the limit.  The
augmented boundary point acquires a second level and as
$\varepsilon \to 0$ it slides into the deeper three-block region. We
are then led to distinguish two pictures:

\begin{enumerate}
  \item the \emph{chamber layer} (walls, binding morphisms, changes
of stable spectrum at finite time)
\item the \emph{region layer} (scale crossovers, level structure, the classification of limits at infinity)
\end{enumerate}

These layers are disjoint, with all the
$\varepsilon$-dependence residing in the second. Keeping them
separate is, we will argue, the correct organization of the
compatibility problem for the fourfold: the chamber theorem sought
there is a statement in the first layer, while the convergence of
the quantum path is a statement in the second.

\subsection*{Terminality for the cubic fourfold}
The perturbed model clarifies what an analogous deformation
would need to be for the cubic fourfold: a direction, in moduli or in
quantum parameters, splitting the rank-24 zero-eigenvalue block into
sub-blocks the flow could then separate. Both preconditions fail
(\S\ref{sec:cubic}):

\begin{enumerate}
  \item Spectrally, the block's rank is constant over the
twenty-dimensional moduli of smooth cubics.  Small quantum
cohomology is a deformation invariant, so the constancy follows from
the spectrum computation at a single point: there is no $\varepsilon$.
  \item Categorically, a phase-separation of the block would exhibit a
semiorthogonal decomposition of $\AX$; but $\AX$ is a connected
Calabi--Yau category, so its Serre functor $S \cong [2]$ converts any
semiorthogonal pair into an orthogonal one and
$\HH^0(\AX) = \CC$ forbids orthogonal splitting.
\end{enumerate}

The resonant model is
precisely the case $\HH^0 = \CC^2$ in which this argument fails,
which has a finer structure. For the cubic
fourfold the coarse region is \emph{terminal}
(Proposition~\ref{prop:terminality}), independent of the choice of path.

We were initially concerned that the rigidity of the atom over moduli renders the
global clauses of \cite[Conjecture~8.1]{R} vacuous, but the two models
above allow us to demonstrate instead that moduli-rigidity of the atom and
dynamical protection of its spectrum are the same fact viewed
from two distinct but related perspectives.

\subsection*{The coupled counterpart}
The Fano models above are, in a precise sense
(\S\ref{subsec:coupled-uncoupled}), the minimal \emph{uncoupled}
theory: the degenerate block carries vanishing pairing, no binding
morphisms, and Barnes-level Riemann--Hilbert data. The analytic portion
of the program requires its own test bed, and we develop it in
\S\ref{sec:A2}: the minimal \emph{coupled} two-class structure, the
$A_2$ quiver category, whose Riemann--Hilbert problem is solved by the
monodromy of the deformed cubic oscillator \cite{BrMa} and whose tau
function is Painlev\'e I.

There, the entire dictionary (periods,
saddle trajectories, the pentagon as the Delabaere--Dillinger--Pham
formula, the deformation datum living on the spectral elliptic curve,
the tau function and its zeros) is explicitly computable, and we
formulate a determinant-line program identifying the tau function with
zeta-regularized spectral data of the operator family.
Section~\ref{sec:detline} computes quantum periods in closed form through order $\hb^3$,
provides a demonstration of all-orders flatness, and examines the curvature current and
elementary factors.

We conduct numerical studies on the symmetric slice
(\S\S\ref{subsec:det-numerics}--\ref{subsec:det-divisor}): the
elementary-factor constant
evaluates to the logarithm of the golden ratio, the
$\mu_5$-assembly of Sibuya-normalised determinants is canonically
normalised, and the tau-divisor identification is verified
nonperturbatively along a Painlev\'e~I trajectory.
The coupled and uncoupled cases
together exhaust the local models of a two-class interaction, and they
calibrate complementary halves of the program: the chamber dictionary
(Fano models) and the analytic dictionary ($A_2$).

\subsection*{What this paper does not do}
We do not prove the chamber theorem for the cubic fourfold. The
results of \cite{KRZ} establish quasi-convergence of lifted quantum
central charge paths for cubic fourfolds not containing a plane and,
for a twisted choice of fundamental solution, produce a path that is
glued at every time and \emph{geometric} at one end
\cite[Thms.~4.6, 4.7 and Thm.~2.72]{KRZ}; what is missing is the
wall-and-chamber analysis between the endpoints.  That is the fourfold
analogue of tracking the torsion sheaves on $\PP^1$.

Instead, we formulate the
expected statement, with the wall structure it requires and a
spectral constraint (the centroid obstruction,
Lemma~\ref{lem:centroid}) that fixes its correct mutation, as
Conjecture~\ref{conj:chamber}; the threefold version,
Conjecture~\ref{conj:threefold}, is stated in the expectation that
it is within reach of current technology.

The analytic study of \S\ref{sec:A2} is carried out here through
its perturbative and numerically checkable layer
(\S\ref{sec:detline}); what remains is an investigation into
the smooth Quillen (spectral-flow and heat-kernel) transfer and the identification
with the JLO cochain.  This is expected to be a more substantial effort
and is planned for a sequel, whose
nonperturbative target is stated as
Conjecture~\ref{conj:jumps}.

That work is our primary
interest, furthering the connections
between algebraic and analytic approaches for spaces of interest
to physical mathematics.

\subsection*{Organization}
Section~\ref{sec:conventions} fixes conventions for quasi-convergent
paths, glued stability conditions, and the saddle-point asymptotics of
quantum central charges, and proves the Sectoriality Lemma
(Lemma~\ref{lem:sectoriality}) used throughout.
Section~\ref{sec:P1} treats $\PP^1$; Section~\ref{sec:product} the
resonant $\PP^1\times\PP^1$; Section~\ref{sec:epsilon} the
$\varepsilon$-deformation, the two clocks, and the wall-splitting of
the resonant walls. Section~\ref{sec:cubic} assembles the terminality
argument for the cubic fourfold, states the chamber theorem for the
cubic threefold as the nearest-term precise conjecture
(Conjecture~\ref{conj:threefold}), formulates the fourfold statement
with its expected wall structure, and records the phrasing
refinements the models show to be necessary.
Section~\ref{sec:A2} develops the coupled case: the $A_2$ structure,
the deformed cubic oscillator, Painlev\'e I, and the determinant
study.
Section~\ref{sec:detline} continues the determinant
to the perturbative level: the first $\hb$-correction of
the determinant--period dictionary in closed form, flatness to all
orders, the exact curvature current, and the reduction of the
elementary-factor matching to a single constant, each statement
verified numerically on the symmetric slice, together with the
tau-divisor comparison along a Painlev\'e~I trajectory.
Section~\ref{sec:outlook} orders the next steps of both programs
and states the nonperturbative jump conjecture
(Conjecture~\ref{conj:jumps})

Readers whose primary interest is the
analytic program can proceed directly to
\S\S\ref{sec:A2}--\ref{sec:detline}, which are self-contained given
the coupled/uncoupled dichotomy of \S\ref{subsec:coupled-uncoupled}.

\section{Conventions and preliminaries}\label{sec:conventions}

Let $X$ denote a smooth projective Fano variety over $\CC$
with bounded derived category $\Db(X)$, and $\Stab(\Db(X))$ the space
of full numerical Bridgeland stability conditions \cite{Bri07}.

For
$\sigma = (Z,\mathcal{P})$ we write $\varphi_\sigma(E) \in \RR$ for
the (lifted) phase of a $\sigma$-semistable object $E$, so that
$Z(E) \in \RR_{>0}\, e^{i\pi\varphi_\sigma(E)}$ and
$\varphi_\sigma(E[1]) = \varphi_\sigma(E) + 1$. This section fixes
the vocabulary; readers familiar with \cite{Bri07,Mac07,HLJR} can
proceed to \S\ref{subsec:paths}.

\subsection{Hearts, stability conditions, and quasi-convergence}
\label{subsec:background}

A \emph{heart} (of a bounded t-structure) in a triangulated category
$\mathcal{D}$ is a full additive subcategory
$\mathcal{A} \subset \mathcal{D}$ such that
$\Hom(A, A'[k]) = 0$ for all $A, A' \in \mathcal{A}$ and $k < 0$,
and such that every object of $\mathcal{D}$ admits a finite
filtration by exact triangles whose graded pieces are shifts
$A_i[k_i]$ with $A_i \in \mathcal{A}$ and
$k_1 > \cdots > k_r$. A heart is an abelian category, with short
exact sequences the triangles of $\mathcal{D}$ having all three
vertices in $\mathcal{A}$; the standard example is
$\operatorname{Coh}(X) \subset \Db(X)$. For objects
$S_1, \dots, S_r$ we write
$\langle S_1, \dots, S_r\rangle_{\mathrm{ext}}$ for their extension
closure, the smallest full subcategory containing the $S_i$ and
closed under extensions.

A \emph{stability condition} $\sigma = (Z, \mathcal{P})$ on
$\mathcal{D}$ \cite{Bri07} consists of a group homomorphism $Z$, where
$Z \colon K_{\mathrm{num}}(\mathcal{D}) \to \CC$ (the
\emph{central charge}) and a \emph{slicing}: full additive
subcategories $\mathcal{P}(\varphi) \subset \mathcal{D}$ for
$\varphi \in \RR$, satisfying
$\mathcal{P}(\varphi + 1) = \mathcal{P}(\varphi)[1]$,
$\Hom(\mathcal{P}(\varphi_1), \mathcal{P}(\varphi_2)) = 0$ for
$\varphi_1 > \varphi_2$, the compatibility
$Z(E) \in \RR_{>0}\, e^{i\pi\varphi}$ for
$0 \neq E \in \mathcal{P}(\varphi)$, and the requirement that
every object have a finite Harder--Narasimhan (HN) filtration with
graded pieces in $\mathcal{P}(\varphi_1), \dots,
\mathcal{P}(\varphi_r)$, $\varphi_1 > \cdots > \varphi_r$.
Objects of $\mathcal{P}(\varphi)$ are the \emph{semistables} of
phase $\varphi$; the simple objects of the (quasi-abelian) category
$\mathcal{P}(\varphi)$ are the \emph{stables}, and every
semistable has a Jordan--H\"older (JH) filtration by stables of the
same phase.

The subcategory
$\mathcal{A} = \mathcal{P}((0,1])$ is a heart, and conversely a
pair $(Z, \mathcal{A})$ with the HN property determines $\sigma$;
crossing a wall of the chamber structure replaces $\mathcal{A}$ by a
tilt at a torsion pair \cite{HRS}. All stability conditions in this
paper satisfy the \emph{support property}: for a fixed norm
$\lVert\cdot\rVert$ on
$K_{\mathrm{num}} \otimes \RR$ there is $c > 0$ with
$|Z(E)| \geq c\,\lVert E \rVert$ for all semistable $E$; this
guarantees a locally finite wall-and-chamber structure and the
deformation theory of \cite{Bri07}. The universal cover
$\widetilde{GL}{}^+_2(\RR)$ of orientation-preserving linear maps
acts on $\Stab(\mathcal{D})$ by post-composition on $Z$ together
with the corresponding relabeling of the slicing; the action
preserves the set of semistable objects. A stability condition on
$\Db(X)$ is \emph{geometric} if all skyscraper sheaves
$\mathcal{O}_x$ are stable of one common phase.

An exceptional collection $E_1, \dots, E_n$
($\RHom(E_i, E_i) = \CC$, $\RHom(E_j, E_i) = 0$ for $j > i$) is
\emph{Ext-exceptional} if moreover
$\Hom^{\leq 0}(E_i, E_j) = 0$ for $i < j$; by Macr\`i
\cite{Mac07}, the extension closure of an Ext-exceptional collection
is a finite-length heart with simple objects the $E_i$. Since a
strong collection has all forward morphisms in degree zero, its
shifts $E_1[a_1], \dots, E_n[a_n]$ form an Ext-exceptional
collection exactly when $a_i - a_j \geq 1$ for every $i < j$ with
$\Hom(E_i, E_j) \neq 0$; this bookkeeping is the source of the
``gap $> 1$'' inequalities appearing throughout the paper
(Appendix~\ref{app:entry}).

Finally, a path $\sigma_t$ in $\Stab(\mathcal{D})$ is
\emph{quasi-convergent} \cite{HLJR} if, qualitatively: every
nonzero object acquires an eventual HN filtration, stable for large
$t$, whose graded pieces have well-defined asymptotic phase slopes
and controlled relative masses; grouping the pieces of each
filtration by slope then assembles, over all objects, into a
semiorthogonal decomposition of $\mathcal{D}$
\emph{induced} by the path. We use only these qualitative features
and refer to \cite{HLJR} for the precise definition.

\subsection{Quantum central charge paths}\label{subsec:paths}
Let $W \colon (\CC^\times)^n \to \CC$ be the Hori--Vafa mirror
superpotential of $X$ and let $\lambda_1,\dots,\lambda_r$ denote its
critical values, \ie\ the eigenvalues of quantum multiplication
$c_1(X)\star$ on $QH^\bullet(X)$. For a class $\gamma \in K_0(X)$
with associated integration cycle $\Gamma_\gamma$ we consider the
central charge
\begin{equation}\label{eq:thimble}
  Z_z(\gamma) \;=\; \int_{\Gamma_\gamma} e^{-W/z}\,
  \frac{dx_1}{x_1}\wedge\cdots\wedge\frac{dx_n}{x_n},
\end{equation}
restricted to the ray
\begin{equation}\label{eq:ray}
  z = e^{i\theta}/t, \qquad t \to +\infty,
\end{equation}
for a fixed angle $\theta$. When $\Gamma_\gamma$ is the Lefschetz
thimble through a nondegenerate critical point $p$ with critical
value $\lambda$, the stationary phase expansion in logarithmic
coordinates gives
\begin{equation}\label{eq:saddle}
  Z_t(\gamma) \;\sim\; (2\pi z)^{n/2}\,
  \bigl(\det \operatorname{Hess}_{\log} W(p)\bigr)^{-1/2}\,
  e^{-\lambda t e^{-i\theta}}\,\bigl(1 + O(1/t)\bigr),
\end{equation}
where $\operatorname{Hess}_{\log}$ denotes the Hessian in the
coordinates $\log x_i$, in which the holomorphic volume form of the
torus is translation-invariant. Two asymptotic invariants are used
constantly:
\begin{align}
  \log|Z_t(\gamma)| &= -t\,\mathrm{Re}\!\left(\lambda e^{-i\theta}\right)
    + O(\log t)
    &&\text{(\emph{mass rate})},\label{eq:massrate}\\
  \pi\,\varphi_t(\gamma) &= -t\,\mathrm{Im}\!\left(\lambda e^{-i\theta}\right)
    + O(1)
    &&\text{(\emph{phase slope})}.\label{eq:phaseslope}
\end{align}
The matching of thimbles to objects is fixed by the
$\widehat{\Gamma}$-integral structure of Iritani \cite{Iri09}; we
record the convention we use in each example, noting that the residual
ambiguity (a simultaneous relabeling) is absorbed into the choice of
ray. Lifts of these central charges to paths $\sigma_t$ in
$\Stab(\Db(X))$, quasi-convergent in the sense of \cite{HLJR}, exist
in all the examples we treat \cite{HL,KRZ}; our purpose is to locate
these paths relative to the chamber structure of $\Stab$, which is
finer information than quasi-convergence.

\subsection{Glued chambers, sectoriality, binding objects}
\label{subsec:glued}
Given a semiorthogonal decomposition
$\Db(X) = \langle \mathcal{D}_1,\dots,\mathcal{D}_\ell\rangle$ (our
convention: $\Hom(\mathcal{D}_j, \mathcal{D}_i) = 0$ for $j > i$) and
stability conditions $\sigma_i$ on the factors, Collins--Polishchuk
\cite{CP} construct, under the vanishing
$\Hom^{\leq 0}(\mathcal{A}_i,\mathcal{A}_j) = 0$ for $i<j$ between
the hearts, a glued stability condition with heart the extension
closure $\langle \mathcal{A}_1,\dots,\mathcal{A}_\ell\rangle$ and
central charge $\bigoplus_i Z_i$. We call the locus of stability
conditions arising this way (and their
$\widetilde{GL}^+_2(\RR)$-translates) the \emph{glued region} of the
decomposition.
When the factors involved are generated by an exceptional
collection, the glued hearts are the Ext-exceptional hearts of
Macr\`i \cite{Mac07}.

In contrast, the \emph{selection region} of the
decomposition is the open sublocus of the glued region on which the
phase windows of the factors are separated in the order of the
decomposition, in the sense of the hypothesis of
Lemma~\ref{lem:sectoriality} below. We call it the selection region
because the dynamical selection rule of \cite[Conj.~8.1]{R} holds
there by construction: sectoriality
(Definition~\ref{def:sectorial} below) is exactly the absence of
stable BPS trajectories between sectors.

We take care to avoid the word
``chamber'': in the models of
\S\S\ref{sec:P1}--\ref{sec:epsilon} the selection region is a
single chamber of the stable-spectrum wall-and-chamber structure,
but for the cubic fourfold it is subdivided by the internal walls of
the Kuznetsov factor \cite{BM14}, across which stability inside
$\AX$ changes while sectoriality persists. The distinction between
the two loci is important to keep in view: the glued region is strictly
larger than the selection region and in general contains \emph{geometric} stability
conditions, with interleaved phase windows.  This overlap is the
mechanism by which geometric stability conditions are constructed on
cubic fourfolds in \cite[Thm.~2.72]{KRZ}, and the geometric versus
non-geometric partition of the glued region is made explicit in
\cite[Fig.~1]{KRZ}.

\begin{definition}\label{def:sectorial}
A stability condition $\sigma$ on $\Db(X)$ is \emph{sectorial} for a
semiorthogonal decomposition
$\langle\mathcal{D}_1,\dots,\mathcal{D}_\ell\rangle$ if every
$\sigma$-stable object lies in some $\mathcal{D}_i$ up to shift, and
every $\sigma$-semistable object has all Jordan--H\"older factors in a
single $\mathcal{D}_i$ up to shift.
\end{definition}

On a measure-zero locus of the selection
region, \emph{semistable} objects of mixed class exist trivially
(direct sums of stables of aligned phase from different factors).
We say an object $E$
is a \emph{binding object} for the decomposition (relative to
$\sigma$) if $E$ is $\sigma$-stable and its class $[E]$ is not
supported on a single factor of
$K_0(\mathcal{D}_1)\oplus\cdots\oplus K_0(\mathcal{D}_\ell)$. Binding
objects are exactly the stable BPS states connecting sectors: a
sectorial $\sigma$ admits no such object.

\subsection{The Sectoriality Lemma}\label{subsec:sectoriality}
We examine the mechanism by which
phase separation forces sectoriality and apply that analysis to both $\PP^1$
(\S\ref{sec:P1}) and the product
(\S\S\ref{sec:product}--\ref{sec:epsilon}).  Then,
with a phase-gap supplied by the asymptotics of the quantum
path, we examine its application to the cubic fourfold (\S\ref{sec:cubic}).

In a semiorthogonal decomposition, every object carries a functorial
filtration whose subobject side is built from the \emph{later}
factors: iterating the projection triangles, $E$ admits a filtration
in which the $\mathcal{D}_\ell$-component appears as a sub, the
$\mathcal{D}_1$-component as a quotient; in a glued heart this
filtration is by subobjects.

\begin{lemma}[Sectoriality Lemma]\label{lem:sectoriality}
Let $\sigma$ be a stability condition in the glued region of
$\langle \mathcal{D}_1, \dots, \mathcal{D}_\ell\rangle$, glued from
$\sigma_i$ on the factors, and suppose the phase windows are separated
in the order of the decomposition: writing $W_i \subseteq (0,1]$ for
the set of phases of $\sigma_i$-semistable objects in the glued heart
$\mathcal{A}$, assume
\[
  \sup W_i \;<\; \inf W_j \qquad \text{for all } i < j
\]
(\ie\ $\sigma$ lies in the selection region).
Then $\sigma$ is sectorial (Definition~\ref{def:sectorial}): every
$\sigma$-stable object lies in a single $\mathcal{D}_i$ up to shift,
every $\sigma$-semistable object is an iterated extension of
$\sigma_i$-semistables of a single $i$ and equal phase, and $\sigma$
admits no binding objects.
\end{lemma}

\begin{proof}
Let $E \in \mathcal{A}$ be $\sigma$-stable with components in at
least two factors, and let $j$ be maximal with
$\mathrm{pr}_j E \neq 0$. The decomposition filtration exhibits a
nonzero subobject $S \subseteq E$ in $\mathcal{A}$ with
$S \in \mathcal{A}_j$ and $E/S$ built from
$\mathcal{A}_1, \dots, \mathcal{A}_{j-1}$. The maximal-phase
Harder--Narasimhan factor of $S$ with respect to $\sigma_j$ furnishes
a subobject $S' \subseteq E$ with $\varphi(S') \geq \inf W_j$. On the
other hand the phases of all Jordan--H\"older constituents of $E/S$
lie in $W_1 \cup \dots \cup W_{j-1}$, hence below $\inf W_j$ by
hypothesis, so $\varphi(E) < \inf W_j \leq \varphi(S')$, contradicting
stability. Thus every stable object lies in a single factor; the
statements for semistables and their Jordan--H\"older factors follow
since a semistable object's factors are stable of equal phase, and
phases from distinct windows differ. Absence of binding objects is
immediate.
\end{proof}

\begin{remark}\label{rem:lemma-scope}
Two features deserve emphasis. First, the lemma requires nothing of
the factors beyond a stability condition. In particular,
$\mathcal{A}_1$ need not be of finite length, which is essential for
the application to the Kuznetsov component in \S\ref{sec:cubic},
where the first factor carries a $K3$-type stability condition with
infinitely many stable classes. Second, the lemma is a statement
\emph{on} the selection region; the complementary assertions in
\S\S\ref{sec:P1}--\ref{sec:product} (that the selection region
exhausts the entire region beyond the entry wall, so that the path
never exits) use the specific vanishing of the reverse-direction
extension groups and are proved there.
\end{remark}

\section{The projective line}\label{sec:P1}

Let $X = \PP^1$ with mirror $W(x) = x + q/x$ on $\CC^\times$. Write
$m = \sqrt{q}$ for a fixed branch. The critical points are $x = \pm m$
with critical values $\lambda_\pm = \pm 2m$ and
$\operatorname{Hess}_{\log}W(\pm m) = \pm 2m$.

\begin{convention}\label{conv:P1}
We match $\mathcal{O} \leftrightarrow \lambda_+ = +2m$ and
$\mathcal{O}(1) \leftrightarrow \lambda_- = -2m$, and fix the ray
\eqref{eq:ray} with $\sin\theta < 0$. Set
\begin{equation}\label{eq:sdef}
  s \;:=\; \frac{4m\,|\sin\theta|}{\pi} \;>\; 0 .
\end{equation}
\end{convention}

By \eqref{eq:phaseslope}, the lifted phase gap
\begin{equation}\label{eq:gap}
  g(t) \;:=\; \varphi_t(\mathcal{O}(1)) - \varphi_t(\mathcal{O})
  \;=\; s\,t + g(0) + O(1/t)
\end{equation}
grows linearly, while by \eqref{eq:massrate} the masses separate at
the exponential rate $4m\cos\theta$ (in whichever direction the sign
of $\cos\theta$ dictates; the phase analysis below is insensitive to
it). We use the complete description of
$\Stab(\PP^1) \cong \CC^2$ due to Okada \cite{Oka06}: one geometric
chamber $G$, on which all skyscrapers $\mathcal{O}_x$ are stable of
equal phase, and algebraic chambers $A_n$, one for each exceptional
pair $(\mathcal{O}(n), \mathcal{O}(n+1))$, with heart the extension
closure of the Ext-exceptional pair
$\{\mathcal{O}(n)[1], \mathcal{O}(n+1)\}$ \cite{Mac07}. The chambers
$A_n$ are precisely the selection regions of the
decompositions $\langle \mathcal{O}(n), \mathcal{O}(n+1)\rangle$ in
the sense of \S\ref{subsec:glued}; the glued regions themselves are
larger, their complementary parts lying inside the geometric chamber
$G$ (on $\PP^1$: the glued heart $\{\mathcal{O}[1],
\mathcal{O}(1)\}$ with gap $g < 1$ has every skyscraper stable ---
the overlap of \cite[Fig.~1]{KRZ}). Okada's classification gives the
global picture: every stability condition on $\PP^1$ is geometric,
algebraic (lying in exactly one $A_n$), or on the single wall
$\{\varphi(\mathcal{O}(n+1)) = \varphi(\mathcal{O}(n)) + 1\}$
separating $G$ from that $A_n$; the chambers are open of full
dimension, the autoequivalence $-\otimes\mathcal{O}(1)$ permutes
the $A_n$, and shifts act by even translations of phase.

\begin{lemma}[Chamber extent for $\PP^1$]\label{lem:P1-extent}
In the region $\{g > 1\}$ the stable objects are exactly
$\mathcal{O}$ and $\mathcal{O}(1)$ up to shift; in particular the
selection region $A_0$ coincides with the entire region $\{g > 1\}$
and every $\sigma$ in it is sectorial for
$\langle \mathcal{O}, \mathcal{O}(1)\rangle$.
\end{lemma}

\begin{proof}
On the selection region, Lemma~\ref{lem:sectoriality} applies (the two
windows are the singleton phases of the two exceptional objects,
separated since $g > 1$), so every stable object lies in
$\langle\mathcal{O}\rangle$ or $\langle\mathcal{O}(1)\rangle$, hence
is a shift of $\mathcal{O}$ or $\mathcal{O}(1)$. It remains to show
the chamber is not exited as $g$ grows: a wall would require a change
of stable spectrum at some phase alignment
$g \in \ZZ_{\geq 2}$, where a new stable would have to be an
extension between the two sectors. An extension whose subobject comes
from the higher-phase sector is never stable; one whose subobject
comes from the lower-phase sector requires a class in
$\Ext^1(\mathcal{O}(1)[j], \mathcal{O}[k]) \cong
H^{k-j+1}(\PP^1, \mathcal{O}(-1)) = 0$. Hence all alignments are
pseudo-walls, and the region $\{g>1\}$ is a single chamber.
\end{proof}

\begin{proposition}[Chamber entry for $\PP^1$]\label{prop:P1-entry}
Assume Convention~\ref{conv:P1} and $g(0) < 1$. Set
$t_0 = \frac{1 - g(0)}{s}\bigl(1 + O(1/t_0)\bigr)$. Then:
\begin{enumerate}
\item For $t < t_0$ the path $\sigma_t$ lies in the geometric chamber
$G$; the binding objects are exactly the torsion sheaves, which are
$\sigma_t$-stable.
\item At $t = t_0$ the path crosses the wall
$\{\varphi(\mathcal{O}(1)) = \varphi(\mathcal{O}) + 1\}$
transversally. On the wall every skyscraper $\mathcal{O}_x$ is
strictly semistable with Jordan--H\"older factors
$\{\mathcal{O}(1), \mathcal{O}[1]\}$, via the triangle
$\mathcal{O}(1) \to \mathcal{O}_x \to \mathcal{O}[1]$ obtained by
rotating
$0 \to \mathcal{O} \to \mathcal{O}(1) \to \mathcal{O}_x \to 0$.
\item For all $t > t_0$ the path lies in the selection region $A_0$; by
Lemma~\ref{lem:P1-extent}, $\sigma_t$ is sectorial and there are no
binding objects.
\end{enumerate}
\end{proposition}

\begin{proof}
(1)--(2): In $G$ the torsion sheaves are stable and constitute the
only stable objects with class mixing $[\mathcal{O}]$ and
$[\mathcal{O}(1)]$ (all classes are
$a[\mathcal{O}] + b[\mathcal{O}_x]$ with
$[\mathcal{O}_x] = [\mathcal{O}(1)] - [\mathcal{O}]$; the stable
objects of $G$ are the line bundles and the torsion sheaves
\cite{Oka06}). The boundary wall of $G$ toward $A_0$ is the phase
alignment
$\varphi(\mathcal{O}(1)) = \varphi(\mathcal{O}_x) =
\varphi(\mathcal{O}[1])$, \ie\ $g = 1$; by \eqref{eq:gap}, $g$ is
strictly increasing with slope $s > 0$, so the wall is reached at the
stated $t_0$ (the $O(1/t_0)$ correction absorbing the subleading
terms of \eqref{eq:phaseslope}) and crossed transversally. The
Jordan--H\"older statement on the wall is immediate from the
displayed triangle and the alignment
$Z(\mathcal{O}_x) = Z(\mathcal{O}(1)) + Z(\mathcal{O}[1])$. (3):
Immediately past the wall the heart tilts at the slope torsion pair
$(\mathcal{T}, \mathcal{F})$ of $\operatorname{Coh}(\PP^1)$ 
into the tilted heart
$\langle \mathcal{F}[1], \mathcal{T}\rangle_{\mathrm{ext}}$
\cite{HRS}.
$\mathcal{T}$ is generated by the torsion sheaves and the line bundles
of degree $\geq 1$, $\mathcal{F}$ by those of degree $\leq 0$ and
the tilted heart corresponds to the Ext-exceptional
heart $\langle\mathcal{O}[1],
\mathcal{O}(1)\rangle_{\mathrm{ext}}$ (both are hearts and one
contains the other), placing $\sigma_t$ in $A_0$
\cite{Oka06,Mac07}; Lemma~\ref{lem:P1-extent} does the rest.
\end{proof}

\begin{remark}[Degenerate rays]\label{rem:stokes-rays}
If $\theta$ is chosen with $\sin\theta = 0$ (a Stokes or anti-Stokes
direction for the pair $\lambda_\pm$), then $s = 0$: the masses
separate but the phase gap is asymptotically constant, and the path
remains in $G$ for all time, converging to a boundary point of the
chamber structure. These are exactly the rays along which the slope
data of \cite{HLJR} degenerates and no semiorthogonal decomposition
can be read off; the quasi-convergence mechanism thus forces the
generic-ray regime in which chamber entry is finite-time. The sign of
$\sin\theta$ (equivalently, rotating $\theta$ through a Stokes
direction) exchanges the limiting decomposition for a mutation,
consistent with the mutation ambiguity of the program.
\end{remark}

\begin{remark}[Protection is achieved along the flow]\label{rem:flow}
The binding BPS states of the pair (the torsion sheaves, one stable
object for each point of $\PP^1$ and each twist) are present at
small $t$ and decouple simultaneously at the single wall $t_0$. The
dynamical selection rule of \cite{R} is, in this model, not a
constraint imposed at the outset but the terminal state of a
wall-crossing: the flow \emph{produces} the sectorial regime.
\end{remark}

\section{The product at resonance}\label{sec:product}

Let $X = \PP^1\times\PP^1$ with mirror
$W(x,y) = x + q_1/x + y + q_2/y$, and impose the resonance
$q_1 = q_2 = q$, $m = \sqrt q$. The critical points are
$(x,y) = (s_1 m, s_2 m)$, $s_i \in \{\pm\}$, with critical values
$\lambda = 2m(s_1 + s_2)$ and
$\det\operatorname{Hess}_{\log}W = 4 s_1 s_2 m^2$:

\begin{center}
\begin{tabular}{cccc}
\hline
vacuum & $\lambda$ & $\det\operatorname{Hess}$ & object \\
\hline
$(+,+)$ & $+4m$ & $+4m^2$ & $\mathcal{O}$ \\
$(-,+)$ & $0$ & $-4m^2$ & $\mathcal{O}(1,0)$ \\
$(+,-)$ & $0$ & $-4m^2$ & $\mathcal{O}(0,1)$ \\
$(-,-)$ & $-4m$ & $+4m^2$ & $\mathcal{O}(1,1)$ \\
\hline
\end{tabular}
\end{center}

\noindent (The matching is the tensor square of
Convention~\ref{conv:P1}, the first factor carrying $q_1$.) The
eigenvalue $0$ of $c_1\star$ is doubly degenerate, while the small
quantum ring $\CC[h_1,h_2]/(h_1^2 - q,\, h_2^2 - q)$ remains
semisimple: only the $c_1$-spectrum coalesces, the situation studied
for Grassmannians by Cotti--Dubrovin--Guzzetti \cite{CDG}. (That
theory presupposes a semisimple ring with coalescing eigenvalues; the
cubic fourfold's degenerate block is \emph{not} of this type, the
ring itself being non-semisimple there; see \S\ref{sec:cubic} and
Remark~\ref{rem:proof-structure}.) We write
$B = \langle \mathcal{O}(1,0), \mathcal{O}(0,1)\rangle$ for the
degenerate block and consider the coarse decomposition
\begin{equation}\label{eq:coarseSOD}
  \Db(\PP^1\times\PP^1)
  \;=\; \bigl\langle\, \mathcal{O},\; B,\;
  \mathcal{O}(1,1)\,\bigr\rangle .
\end{equation}
Complete details for every computation quoted in this section and in
\S\ref{sec:epsilon}, including critical data, saddle expansions, the
symmetry argument, the K\"unneth tables, the gluing inequalities,
the wall analysis, and the deformation, are collected in
Appendix~\ref{app:product}.

\subsection{Exact consequences of the swap symmetry}

\begin{lemma}\label{lem:swap}
At $q_1 = q_2$, the involution $\tau\colon (x,y)\mapsto(y,x)$
preserves $W$ and exchanges the two $\lambda = 0$ thimbles;
correspondingly, the swap automorphism of $\PP^1\times\PP^1$
exchanges $\mathcal{O}(1,0) \leftrightarrow \mathcal{O}(0,1)$ and one
has, identically in $t$,
\begin{equation}\label{eq:locking}
  Z_t(\mathcal{O}(1,0)) \;=\; Z_t(\mathcal{O}(0,1)).
\end{equation}
Consequently, for
$\delta := [\mathcal{O}(1,0)] - [\mathcal{O}(0,1)]$ (so
$\mathrm{ch}(\delta) = h_1 - h_2$):
\begin{enumerate}
\item $Z_t(\delta) \equiv 0$: the central charge factors through the
rank-three quotient $K_0/\langle\delta\rangle$, and no
$\sigma_t$-semistable object can have class in
$\ZZ\delta\setminus\{0\}$;
\item the two objects $\mathcal{O}(1,0), \mathcal{O}(0,1)$ are stable
of \emph{equal} phase for every $t$: the block is exactly
phase-locked;
\item no ray \eqref{eq:ray} produces a relative phase drift within
$B$: the internal channel has no clock.
\end{enumerate}
\end{lemma}

\begin{proof}
Functoriality of \eqref{eq:thimble} under the automorphism
(equivalently, the substitution $(x,y)\mapsto(y,x)$ in the integral)
gives \eqref{eq:locking}; (1)--(3) are immediate. For the
parenthetical claim in (1): a semistable object of class $k\delta$
would have $Z_t$-mass zero. (The obvious object of class $\delta$,
namely $\mathcal{O}(1,0)\oplus\mathcal{O}(0,1)[1]$, has summands of
phases $\varphi$ and $\varphi+1$ and is never semistable,
consistently.)
\end{proof}

\subsection{Homological bookkeeping}

\begin{lemma}\label{lem:ext}
The collection
$\langle \mathcal{O}, \mathcal{O}(1,0), \mathcal{O}(0,1),
\mathcal{O}(1,1)\rangle$ is full and strong, with forward morphisms
concentrated in degree $0$:
\[
  \Hom(\mathcal{O}, \mathcal{O}(1,0)) =
  \Hom(\mathcal{O},\mathcal{O}(0,1))
  = \Hom(\mathcal{O}(1,0),\mathcal{O}(1,1))
  = \Hom(\mathcal{O}(0,1),\mathcal{O}(1,1)) = \CC^2,
\]
and $\Hom(\mathcal{O},\mathcal{O}(1,1)) = \CC^4$. Moreover
\emph{every} backward $\RHom$ vanishes entirely, and the middle pair
is completely orthogonal:
\[
  \RHom(\mathcal{O}(1,0), \mathcal{O}(0,1)) \;=\;
  \RHom(\mathcal{O}(0,1), \mathcal{O}(1,0)) \;=\; 0 .
\]
In particular
$B \cong \Db(\mathrm{pt}) \oplus \Db(\mathrm{pt})$, with no mutation
required, and \eqref{eq:coarseSOD} is semiorthogonal with the block
$B$ an orthogonal direct sum.
\end{lemma}

\begin{proof}
K\"unneth. Every backward twist $\mathcal{O}(a,b)$ occurring has
$a = -1$ or $b = -1$, hence contains a tensor factor
$H^\bullet(\PP^1, \mathcal{O}(-1)) = 0$; the two displayed middle
$\RHom$s are $H^\bullet(\mathcal{O}(-1,1))$ and
$H^\bullet(\mathcal{O}(1,-1))$, each containing such a factor.
Forward $\RHom$s are $H^\bullet(\mathcal{O}(a,b))$ with
$a,b\in\{0,1\}$, hence concentrated in degree $0$ with the stated
dimensions.
\end{proof}

\subsection{Asymptotics and chamber entry}

By \eqref{eq:saddle} with $n = 2$:
\begin{center}
\begin{tabular}{lll}
\hline
object & mass $|Z_t|$ & lifted phase drift \\
\hline
$\mathcal{O}$ & $(\pi/mt)\, e^{-4mt\cos\theta}$ &
  $+(4m\sin\theta/\pi)\,t$\\
$\mathcal{O}(1,0),\ \mathcal{O}(0,1)$ & $\pi/mt$ \ (equal, exactly) &
  $0$ \ (constant $+\,O(1/t)$)\\
$\mathcal{O}(1,1)$ & $(\pi/mt)\, e^{+4mt\cos\theta}$ &
  $-(4m\sin\theta/\pi)\,t$\\
\hline
\end{tabular}
\end{center}

\noindent Three mass scales appear: exponentially light, polynomially
massless, exponentially heavy. With $\sin\theta < 0$ as in
Convention~\ref{conv:P1}, the phase slopes order
$\mathcal{O} < B < \mathcal{O}(1,1)$, matching the semiorthogonal
order of \eqref{eq:coarseSOD}; the opposite sign would demand the
reversed order, which is not semiorthogonal
($\Hom(\mathcal{O},\mathcal{O}(1,1)) \neq 0$), and produces a mutated
collection instead. Define the adjacent-block gaps
\[
  \gamma_1(t) = \varphi_t(B) - \varphi_t(\mathcal{O}), \qquad
  \gamma_2(t) = \varphi_t(\mathcal{O}(1,1)) - \varphi_t(B),
\]
each of slope $s$ as in \eqref{eq:sdef} (the phase $\varphi_t(B)$
being well defined by Lemma~\ref{lem:swap}(2)).

\begin{lemma}[Chamber extent for the product]\label{lem:prod-extent}
In the region $\{\gamma_1 > 1,\ \gamma_2 > 1\}$ the stable objects
are exactly the four exceptional objects up to shift, with the middle
pair of equal phase; every $\sigma$ there is sectorial for
\eqref{eq:coarseSOD}.
\end{lemma}

\begin{proof}
By Lemma~\ref{lem:ext} all backward $\RHom$s vanish, so the
Collins--Polishchuk condition for \eqref{eq:coarseSOD} reduces to the
two phase inequalities $\gamma_1 > 1$, $\gamma_2 > 1$ (the skip
condition for $\Hom(\mathcal{O},\mathcal{O}(1,1))$ being implied),
and within $B$ it is vacuous by complete orthogonality. On the
selection region, Lemma~\ref{lem:sectoriality} applied to the
three-factor
decomposition \eqref{eq:coarseSOD} gives sectoriality; the stables of
the outer factors are shifts of the corresponding exceptional
objects, and those of the orthogonal sum $B$ are its two simples
(direct sums are excluded). Chamber extent: as in
Lemma~\ref{lem:P1-extent}, a wall would require a stable extension
between sectors with subobject from the lower-phase sector, hence a
class in a backward $\Ext^1$, all of which vanish by
Lemma~\ref{lem:ext}; all phase alignments are pseudo-walls.
\end{proof}

\begin{proposition}[Chamber entry at resonance]\label{prop:prod-entry}
Assume Convention~\ref{conv:P1} and
$\gamma_1(0), \gamma_2(0) < 1$, and set
\[
  t_0 \;=\; \frac{1}{s}\,
  \max\bigl(1-\gamma_1(0),\; 1-\gamma_2(0)\bigr)
  \,\bigl(1+O(1/t_0)\bigr).
\]
Then for $t > t_0$ the path $\sigma_t$ lies in the selection region
of \eqref{eq:coarseSOD}; by Lemma~\ref{lem:prod-extent} it is sectorial,
its stable spectrum is the four exceptional objects up to shift with
the middle pair exactly phase-locked, and it admits no binding
objects. Entry occurs through (at most) two walls
$\{\gamma_i = 1\}$, and each is a wall \emph{of multiplicity two}: at
$\gamma_1 = 1$ the line sheaves $\mathcal{O}_{D_1}$ and
$\mathcal{O}_{D_2}$, for $D_1 \in |h_1|$, $D_2 \in |h_2|$, of
classes $[\mathcal{O}(1,0)]-[\mathcal{O}]$ and
$[\mathcal{O}(0,1)]-[\mathcal{O}]$, become strictly semistable
\emph{simultaneously}, with Jordan--H\"older factors
$\{\mathcal{O}(1,0), \mathcal{O}[1]\}$ and
$\{\mathcal{O}(0,1), \mathcal{O}[1]\}$ respectively, via rotation of
$0 \to \mathcal{O} \to \mathcal{O}(1,0) \to \mathcal{O}_{D_1} \to 0$
and its mirror image; likewise at $\gamma_2 = 1$ for the twisted line
sheaves of classes $[\mathcal{O}(1,1)]-[\mathcal{O}(1,0)]$ and
$[\mathcal{O}(1,1)]-[\mathcal{O}(0,1)]$. The simultaneity is forced
by the exact locking \eqref{eq:locking} and unfolds under
perturbation (\S\ref{sec:epsilon}).
\end{proposition}

\begin{proof}
Both gaps have slope $s > 0$, giving the entry time; transversal
crossing and the identification of the walls with the stated
destabilizations follow as in
Proposition~\ref{prop:P1-entry}(2), and everything past entry is
Lemma~\ref{lem:prod-extent}.
\end{proof}

\begin{remark}[The limit is a massless block]\label{rem:massless}
Along the path, $|Z_t(B)| \sim \pi/mt \to 0$ polynomially while the
outer sectors separate exponentially: after any overall
normalization, the middle block goes massless relative to the heavy
sector. The limit of $\sigma_t$ is thus not an interior point of
$\Stab$ but a boundary point of the augmented stability space
\cite{HLJR,RobAug} at which the block $B$ carries degenerate central
charge.  This is the multi-scale structure of \S\ref{sec:epsilon} at its
most degenerate. For every finite $t$, $\sigma_t$ is then an
interior stability condition.
\end{remark}

\begin{remark}[Skyscrapers and higher binding classes]
\label{rem:koszul}
The skyscraper $\mathcal{O}_p$, of class
$[\mathcal{O}]-[\mathcal{O}(1,0)]-[\mathcal{O}(0,1)]
+[\mathcal{O}(1,1)]$, mixes all sectors; its eventual
Harder--Narasimhan factors are those of the twisted Koszul complex,
$\{\mathcal{O}(1,1),\ (\mathcal{O}(1,0)\oplus\mathcal{O}(0,1))[1],\
\mathcal{O}[2]\}$. More generally each adjacent channel carries
infinitely many binding classes (the twists $\mathcal{O}_{D}(k)$,
$k\in\ZZ$). We emphasize the structural point: the boundary of the
selection region is the failure locus of the
\emph{finitely many} gluing inequalities
of Lemma~\ref{lem:prod-extent}, and once these hold every binding
class is automatically filtered; the infinite enumeration of
individual binding walls lives strictly on the geometric side of the
boundary. We do not undertake here the (finite-time, per-object) wall
bookkeeping for the $\mathcal{O}_D(k)$ and $\mathcal{O}_p$, which
belongs to the wall analysis of the general program rather than to
the entry statement.
\end{remark}

\section{The perturbed resonance}\label{sec:epsilon}

We now switch on the deformation
\[
  q_1 = q\, e^{i\varepsilon}, \qquad q_2 = q, \qquad
  0 < \varepsilon \ll 1,
\]
which does not move the category (Lemma~\ref{lem:ext} is
$q$-independent) but unfolds the spectral resonance
(details in Appendix~\ref{app:eps}). The critical
values become $\lambda = 2m(s_1 e^{i\varepsilon/2} + s_2)$; in
particular the middle splitting is, exactly,
\begin{equation}\label{eq:splitting}
  \Delta\lambda_B \;:=\; \lambda(\mathcal{O}(0,1)) -
  \lambda(\mathcal{O}(1,0)) \;=\; 4m\bigl(e^{i\varepsilon/2}-1\bigr)
  \;=\; 2im\varepsilon \;-\; \tfrac{1}{2}m\varepsilon^2 \;+\;
  O(\varepsilon^3),
\end{equation}
purely imaginary at leading order.

\subsection{Two clocks}
By \eqref{eq:massrate}--\eqref{eq:phaseslope} applied to
\eqref{eq:splitting}, the internal separation of the block is
governed by two rates:
\begin{align}
  s_f &\;=\; \frac{1}{\pi}\,
  \Bigl|\, 2m\varepsilon\cos\theta +
  \tfrac{1}{2}m\varepsilon^2 \sin\theta \,\Bigr| + O(\varepsilon^3)
  && \text{(\emph{phase clock})},\label{eq:phaseclock}\\
  r_f &\;=\; m\varepsilon\,|\sin\theta| + O(\varepsilon^2)
  && \text{(\emph{mass clock})},\label{eq:massclock}
\end{align}
computed from
$\mathrm{Re}(\pm im\varepsilon\, e^{-i\theta}) =
\pm m\varepsilon\sin\theta$. The two clocks are weighted by
\emph{complementary} projections of the splitting onto the ray: the
phase clock runs on $\cos\theta$, the mass clock on $\sin\theta$.
The coarse rates of \S\ref{sec:product} are unchanged to
$O(\varepsilon)$. Consequently the internal gap
$g_B(t) = \varphi_t(\mathcal{O}(1,0)) -
\varphi_t(\mathcal{O}(0,1))$ (the sign of $\varepsilon$ fixing the
orientation, and thereby selecting between the two
fine orderings of the middle pair) reaches order
one at the \emph{crossover time}
\begin{equation}\label{eq:t1}
  t_1 \;=\; \frac{\pi}{2m\varepsilon\,|\cos\theta|}\,
  \bigl(1 + O(\varepsilon)\bigr)
  \quad (\cos\theta \neq 0),
  \qquad
  t_1 \;=\; \frac{2\pi}{m\varepsilon^2 |\sin\theta|}\,
  \bigl(1 + O(\varepsilon)\bigr)
  \quad (\cos\theta = 0),
\end{equation}
the second formula arising because at $\cos\theta = 0$ the leading
imaginary part of \eqref{eq:splitting} drops out of the phase
projection and the $O(\varepsilon^2)$ real part takes over. Even the
scale hierarchy of the crossover is thus ray-dependent.

\subsection{Wall splitting}
The multiplicity-two walls of Proposition~\ref{prop:prod-entry}
unfold: since $g_B(t) = s_f\, t + O(\varepsilon)$ is already nonzero
at the coarse walls, the deaths of $\mathcal{O}_{D_1}$ and
$\mathcal{O}_{D_2}$ occur at times differing by
$\Delta t = \bigl(s_f\, t_0 + O(\varepsilon)\bigr)/s =
O(\varepsilon)$, and likewise in the second channel: four coarse
walls in two $\varepsilon$-close pairs, the wall multiplicity of the
resonance unfolding like an eigenvalue crossing.

\subsection{The crossover is not a wall}

\begin{proposition}\label{prop:not-a-wall}
For every $\varepsilon$ (including $\varepsilon = 0$) and every
$t > t_0$, the stable spectrum of $\sigma_t$ is the four exceptional
objects up to shift, and $\sigma_t$ is sectorial for the \emph{fine}
decomposition $\langle \mathcal{O}, \mathcal{O}(0,1),
\mathcal{O}(1,0), \mathcal{O}(1,1)\rangle$ as well as for the coarse
one. In particular no change of stable spectrum, and no wall, occurs
at the crossover time $t_1$ of \eqref{eq:t1}: objects of mixed middle
class are direct sums $E_1 \oplus E_2$ with $E_i$ in the two
orthogonal summands of $B$, semistable only at the discrete times
where $g_B(t) \in 2\ZZ$, and never stable.
\end{proposition}

\begin{proof}
Complete orthogonality of the middle pair (Lemma~\ref{lem:ext}, valid
for all $q_i$) makes the Collins--Polishchuk condition between its
factors vacuous and forbids indecomposable objects of mixed middle
class; the remaining assertions repeat
Lemma~\ref{lem:prod-extent} and
Proposition~\ref{prop:prod-entry}, whose hypotheses are insensitive
to $\varepsilon$ at the orders computed.
\end{proof}

What changes at $t_1$ is not the chamber but the \emph{region}: for
$t_0 \ll t \ll t_1$ the configuration is metrically three-block (the
internal gap is $\ll 1$), for $t \gg t_1$ four-block. The limit of
the path is a two-level point of the augmented boundary:
\begin{enumerate}
\item The coarse separation \eqref{eq:coarseSOD} at rate $s$
\item The internal separation at rate $s_f$
\end{enumerate}
with the ratio
$s_f/s = (\varepsilon/2)|\cot\theta| + O(\varepsilon^2)$.  As
$\varepsilon \to 0$ the second level collapses and the limit point
moves into the deeper three-block region.

\begin{remark}[Two layers]\label{rem:layers}
  The analysis separates into two independent layers:

  \begin{enumerate}
\item The
\emph{chamber picture} (walls, binding morphisms, changes of stable
spectrum) is a finite-time, stability-theoretic structure; in this
model it is entirely coarse and entirely $\varepsilon$-robust
(Proposition~\ref{prop:not-a-wall}).
\item The \emph{region picture} (scale
crossovers, level structure, classification of limits) is an
asymptotic, augmented-boundary structure \cite{HLJR,RobAug}; in this
model it carries all the $\varepsilon$-dependence, including the
ray-dependent degeneration type visible in
\eqref{eq:phaseclock}--\eqref{eq:massclock}.
  \end{enumerate}
  
The compatibility program for the cubic fourfold should be
organized with this separation in place: the chamber theorem sought
there is a statement in the first layer; the convergence of the
quantum path is a statement in the second.
\end{remark}

\begin{remark}[Coupled and uncoupled pairs]\label{rem:coupled-preview}
The middle block of \S\ref{sec:product} is the minimal
\emph{uncoupled} two-class degeneracy: vanishing pairing, no binding
morphism, no wall, and all phenomena in the region picture. Its
coupled counterpart --- nonzero pairing, a genuine two-sided wall,
and transcendental rather than solvable Riemann--Hilbert data --- is
the $A_2$ structure developed in \S\ref{sec:A2}, and the two together
exhaust the local models of a two-class interaction.
\end{remark}

\section{The cubic fourfold: terminality and the chamber conjecture}
\label{sec:cubic}

We now assemble what the models imply for the motivating case. Let
$X \subset \PP^5$ be a smooth cubic fourfold, with Kuznetsov
decomposition
\begin{equation}\label{eq:kuzSOD}
  \Db(X) \;=\; \bigl\langle\, \AX,\; \mathcal{O}_X,\;
  \mathcal{O}_X(1),\; \mathcal{O}_X(2)\,\bigr\rangle,
\end{equation}
where $\AX$ is a connected Calabi--Yau category of dimension two,
$S_{\AX} \cong [2]$ \cite{Kuz10,KuzCY}, with Mukai lattice $\Mukai$
\cite{AT}. Bridgeland stability conditions on $\AX$ exist and satisfy
a support property controlled by the $A_2$-lattice \cite{BLMS}; for
very general $X$ one has $K_{\mathrm{num}}(\AX) \cong A_2$. On the
spectral side, quantum multiplication $c_1(X)\star$ on
$QH^\bullet(X)$ has three simple nonzero eigenvalues
$u_0, u_1, u_2 = 9, 9e^{\pm 2\pi i/3}$, permuted cyclically by the
Galois symmetry of the quantum parameter, together with the
eigenvalue $0$ of multiplicity $24 = \dim H^\bullet(X) - 3$, which
carries the $K3$ atom \cite{KKPY}. Concretely, the ambient quantum
ring is $(H^\bullet_{\mathrm{amb}}(X), \star_0) \cong
\CC[h]/(h^2(h^3 - 27))$ \cite{SS}, so the $0$-eigenspace is the sum
of a two-dimensional \emph{Jordan block} on the ambient part and the
$22$-dimensional primitive cohomology (see \cite[\S 3.4]{KRZ}): the
atom's block is not semisimple even before the primitive part is
added. Lifted quantum central charge paths, quasi-convergent
with limiting decomposition \eqref{eq:kuzSOD}, are constructed for
$X$ not containing a plane in \cite{KRZ}.

\subsection{Terminality of the coarse regions}

The perturbed model of \S\ref{sec:epsilon} clarifies what a finer
region picture for the cubic fourfold would require: a deformation (of
the variety in its moduli, or of the quantum parameters) splitting
the rank-$24$ eigenvalue-zero block into sub-blocks whose phases a
ray could then separate, together with the categorical shadow of that
splitting inside $\AX$. Both requirements fail.

\begin{lemma}\label{lem:CY-indec}
Let $\mathcal{C}$ be a triangulated category with Serre functor
$S_{\mathcal{C}} \cong [n]$ for some $n \in \ZZ$ and
$\HH^0(\mathcal{C}) = \CC$. Then $\mathcal{C}$ admits no nontrivial
semiorthogonal decomposition. In particular $\AX$ (with $n = 2$)
admits none.
\end{lemma}

\begin{proof}
Suppose $\mathcal{C} = \langle \mathcal{B}, \mathcal{C}'\rangle$ with
$\Hom(\mathcal{C}', \mathcal{B}) = 0$. Serre duality gives, for
$b \in \mathcal{B}$, $c \in \mathcal{C}'$,
\[
  \Hom(b, c) \;\cong\; \Hom\bigl(c, S_{\mathcal{C}}\, b\bigr)^\vee
  \;=\; \Hom\bigl(c, b[n]\bigr)^\vee \;=\; 0,
\]
so the decomposition is fully orthogonal:
$\mathcal{C} = \mathcal{B} \oplus \mathcal{C}'$. An orthogonal
decomposition splits the identity functor, whence
$\HH^0(\mathcal{C}) \cong \HH^0(\mathcal{B}) \oplus
\HH^0(\mathcal{C}')$ contains two orthogonal idempotents,
contradicting $\HH^0(\mathcal{C}) = \CC$ unless one factor vanishes.
For $\AX$, connectedness $\HH^0(\AX) = \CC$ is part of its structure
as a connected CY2 category \cite{KuzCY}. (In specific cases the
$\HH^0$-hypothesis can be bypassed: for the $A_2$ category of
\S\ref{sec:A2}, an orthogonal splitting would separate the simples
against $\Ext^1(S_1,S_2)\neq 0$.)
\end{proof}

\begin{proposition}[Terminality]\label{prop:terminality}
For every smooth cubic fourfold $X$ and every ray \eqref{eq:ray}:
\begin{enumerate}
\item \textup{(no $\varepsilon$-direction)} The spectrum of
$c_1\star$ consists of three simple nonzero eigenvalues and $0$ with
multiplicity $24$.  It is the same for every smooth cubic fourfold: the
small quantum product is assembled from genus-zero Gromov--Witten
invariants, which are deformation invariants, and smooth cubic
fourfolds form a single deformation class, so the constancy follows
from the computation at a single (very general) point
\cite[\S 6.2]{KKPY}, available also through the ambient quantum ring
\cite{SS}, \cite[\S 3.4]{KRZ} (see \cite{HYZZ} for the uniqueness of
the spectral decomposition of the associated $F$-bundle). No
deformation of $X$ unfolds the spectral degeneracy carrying the
atom.
\item \textup{(no categorical splitting)} In the multi-scale
formalism of \cite{HLJR,RobAug}, a refinement of the limiting
decomposition \eqref{eq:kuzSOD} within the factor $\AX$ would induce
a semiorthogonal decomposition of $\AX$, which does not exist by
Lemma~\ref{lem:CY-indec}.
\end{enumerate}
Consequently the coarse region picture, the multi-scale decomposition
\eqref{eq:kuzSOD} with the block $\AX$ unrefined, is terminal: the
two-level structure of \S\ref{sec:epsilon}, and the crossover time
$t_1$, have no analogue for the cubic fourfold.
\end{proposition}

\begin{proof}
(1) is proved in the statement; (2) is Lemma~\ref{lem:CY-indec}
combined with the fact that multi-scale refinements of a factor are
semiorthogonal sequences of admissible subcategories of that factor
\cite{HLJR,RobAug}. The final statement combines the two: a finer
region picture would require a splitting of the eigenvalue-zero block
(excluded by (1)) inducing a decomposition of $\AX$ (excluded by
(2)).
\end{proof}

\begin{remark}[The resonant model as the boundary case]
\label{rem:boundary-case}
The product at resonance (\S\ref{sec:product}) is the
situation in which both clauses of
Proposition~\ref{prop:terminality} fail: the block $B$ has
$\HH^0(B) = \CC^2$ (cf. the assumptions of Lemma~\ref{lem:CY-indec})
and the deformation
$q_1/q_2 = e^{i\varepsilon}$ can be viewed as a spectral direction.
\end{remark}

\begin{remark}[Rigidity is protection]\label{rem:rigidity}
Proposition~\ref{prop:terminality} inverts a deflationary reading of
the rigidity results surrounding the atom formalism. One might worry
that the constancy of the atom over moduli renders the global,
monodromy-variation clauses of \cite[Conjecture~8.1]{R} vacuous as
nothing varies. The two-layer analysis of \S\ref{sec:epsilon} shows
the moduli-rigidity of the atom (clause (1))  and dynamical protection
of the spectrum (terminality;  absence of flows
resolving, splitting, or coupling to the block's interior) are
the same observation expressed in different technical language (i.e.,
chambers vs. regions).
\end{remark}

\subsection{The cubic threefold}
\label{subsec:threefold}

Before stating the fourfold conjecture we record its dimension-three
counterpart. Let $Y \subset \PP^4$ be a smooth cubic threefold, with
Kuznetsov decomposition
$\Db(Y) = \langle \mathrm{Ku}(Y), \mathcal{O}, \mathcal{O}(1)\rangle$.
The component $\mathrm{Ku}(Y)$ is a connected \emph{fractional}
Calabi--Yau category, $S^3 \cong [5]$ \cite{KuzCY}, with
$K_{\mathrm{num}} \cong \ZZ^2$; by Feyzbakhsh--Pertusi its stability
manifold consists of a \emph{single}
$\widetilde{GL}{}^+_2(\RR)$-orbit \cite{FP}. On the spectral side,
$(H^\bullet_{\mathrm{amb}}(Y), \star_0) \cong
\CC[h]/(h^2(h^2 - 27))$, so $c_1(Y)\star = 2h$ has the two simple
eigenvalues $\pm T$, $T = 2\sqrt{6}$, and the eigenvalue $0$ of
multiplicity $12 = 2 + 10$: a two-dimensional Jordan block on the
ambient part together with the ten-dimensional \emph{odd} primitive
cohomology $H^3(Y)$ \cite{SS}, \cite[\S 4.2]{KRZ}. The atom of the
threefold is thus the intermediate-Jacobian atom, and the protection
statement below is the categorical shadow of Clemens--Griffiths.

\begin{lemma}[Centroid obstruction]\label{lem:centroid}
The nonzero eigenvalues of $c_1(X)\star$ on a cubic fourfold satisfy
$u_0 + u_1 + u_2 = 9(1 + \zeta_3 + \zeta_3^2) = 0$: the eigenvalue
$0$ of the atom is their centroid. Consequently, on every ray
avoiding the degenerate set of Conjecture~\ref{conj:chamber}, the
phase slopes $-\operatorname{Im}(u_i e^{-i\theta})/\pi$ of the three
exceptional sectors are nonzero, pairwise distinct, and sum to zero,
so at least one is positive and one is negative; the slope $0$ of
$\AX$ is \emph{strictly interior} to the exceptional slopes. In
particular no ray realizes an ordering with $\AX$ first or last: for
every ray and every smooth cubic fourfold, the Kuznetsov component
occupies an interior position of the induced ordering. The same holds
for cubic threefolds ($u_\pm = \pm T$, with $\mathrm{Ku}(Y)$ always
in the middle position).
\end{lemma}

\begin{proof}
Immediate: if all three slopes had the same sign their sum could not
vanish, and on the non-degenerate rays no slope vanishes and no two
coincide.
\end{proof}

\begin{remark}[Answering a question of \cite{KRZ}]\label{rem:centroid}
Lemma~\ref{lem:centroid} resolves an uncertainty stated explicitly in
\cite[\S 4.4]{KRZ}: there, geometric stability conditions are
constructed in the glued region of the Kuznetsov ordering
\eqref{eq:kuzSOD}, while the natural quantum path induces a mutation
with $\AX$ in the middle \cite[Thm.~4.6]{KRZ} (for threefolds,
$\langle \mathcal{O}(1), \mathrm{Ku}(Y), \mathcal{O}(2)\rangle$
\cite[Thm.~4.4]{KRZ}), and the authors ask whether this mismatch is a
technical limitation or a structural feature of the Kuznetsov
decomposition.

It is structural. Since the atom's eigenvalue is the
centroid of the exceptional eigenvalues, no choice of ray for any
fundamental solution with the canonical eigenvalue assignment can
place the Kuznetsov component at an end of the ordering. Realizing
the Kuznetsov order requires twisting the fundamental solution, as in
\cite[Thms.~4.5, 4.7]{KRZ}, or a genuine isomonodromic deformation
(\cite[Question~4.8]{KRZ}; see
Remark~\ref{rem:proof-structure}).
\end{remark}

\begin{conjecture}[Chamber theorem for the cubic threefold]
\label{conj:threefold}
Let $Y$ be a smooth cubic threefold and $(\sigma_t)$ a lifted quantum
central charge path as in \cite[Thms.~4.4, 4.5]{KRZ}, along a ray
$\theta$ with $\operatorname{Im}(Te^{-i\theta}) \neq 0$, inducing by
Lemma~\ref{lem:centroid} the interior-$\mathrm{Ku}$ ordering
$\mathbb{D}_\theta = \langle \mathcal{O}(1), \mathrm{Ku}(Y),
\mathcal{O}(2)\rangle$ (up to mutation). Then there exists
$t_0 < \infty$ such that for $t > t_0$ the path lies in the selection
region of $\mathbb{D}_\theta$, glued from the unique-orbit stability
condition on $\mathrm{Ku}(Y)$ \cite{FP} and Ext-exceptional
conditions on the outer factors; the boundary of the selection region
along the path is the failure locus of the two adjacent gluing
inequalities, with entry clock $t_0^{-1} \sim
T\,|\operatorname{Im}(e^{-i\theta})|$ and offsets given by
\cite[(4.5)]{KRZ}; the entry walls are strict semistabilizations of
binding objects with classes $v + w$, $v \in
K_{\mathrm{num}}(\mathrm{Ku}(Y))$, $w$ exceptional, with
sectorial Jordan--H\"older factors given by the projection triangles
of \cite[\S 2]{KRZ}. For $t > t_0$ there are no binding objects:
the spectrum of the intermediate-Jacobian atom is dynamically
protected.
\end{conjecture}

\begin{remark}[On dimension three]
\label{rem:threefold-reach}
Each ingredient of the architecture recorded in
Remark~\ref{rem:proof-structure} below simplifies in dimension
three. The endpoints exist: a geometric glued condition and a
quasi-convergent glued path joining it to the sectorial tail are
\cite[Thm.~2.55]{KRZ} and \cite[Thm.~4.5]{KRZ}. The offsets are the
explicit constants of \cite[(4.5)]{KRZ}.

Only two gluing
inequalities occur, with thresholds bounded through the Serre
duality $\RHom(A, \mathcal{O}(i)) \cong
\RHom(\mathcal{O}(i+2), A[3])^\vee$. Also, as
$\Stab(\mathrm{Ku}(Y))$ is a single orbit
\cite{FP}, the caveat of \S\ref{subsec:glued} is vacuous: the
selection region \emph{is} an honest chamber, with no internal
walls, in contrast to the fourfold case, where the Bayer--Macr\`i walls
of $\AX$ subdivide it.
\end{remark}

\begin{remark}[Terminality in dimension three]
\label{rem:threefold-terminality}
The fourfold's terminality argument does not transfer verbatim:
$\mathrm{Ku}(Y)$ is fractional Calabi--Yau, $S^3 \cong [5]$, so
Lemma~\ref{lem:CY-indec} (which requires $S \cong [n]$) does not
apply. The spectral clause survives unchanged. The multiplicity-%
$12$ block is constant over the moduli of smooth cubic threefolds by
deformation invariance of the small quantum product. The
categorical clause rests instead on the indecomposability of
$\mathrm{Ku}(Y)$, which contains no exceptional objects and is
expected to admit no nontrivial semiorthogonal decomposition
(cf.\ \cite{KuzCY}).  Assuming this, the coarse region is terminal in
dimension three as well. The asymmetry is itself informative: the
clean Serre-duality argument of Lemma~\ref{lem:CY-indec} is a
special feature of the CY2 case.
\end{remark}

\subsection{The fourfold chamber conjecture}

\begin{conjecture}[Chamber theorem for the cubic fourfold]
\label{conj:chamber}
Let $X$ be a smooth cubic fourfold not containing a plane, and let
$(\sigma_t)_{t \geq 0}$ be a lifted quantum central charge path as in
\cite{KRZ}, along a ray $\theta$ avoiding the finite degenerate set
\[
  \bigl\{\theta : \operatorname{Im}\bigl((u_i - u_j)e^{-i\theta}\bigr)
  = 0 \text{ for some } i \neq j \bigr\}
  \;\cup\;
  \bigl\{\theta : \operatorname{Im}\bigl(u_i e^{-i\theta}\bigr) = 0
  \text{ for some } i \bigr\}.
\]
By Lemma~\ref{lem:centroid} the induced slope ordering places $\AX$
in an interior position; write $\mathbb{D}_\theta$ for the
corresponding mutation of \eqref{eq:kuzSOD} --- for the reference
rays of \cite[Thm.~4.6]{KRZ},
$\mathbb{D}_\theta = \langle \mathcal{O}(2),
\mathrm{R}_{\mathcal{O}(2)}\mathcal{O},\, \AX,\,
\mathrm{R}_{\mathcal{O}(2)}\mathcal{O}(1)\rangle$, where
$\mathrm{R}_E$ denotes right mutation through $E$. Then:
\begin{enumerate}
\item \textup{(Entry)} There exists $t_0 < \infty$ such that for all
$t > t_0$ the path lies in the selection region of
$\mathbb{D}_\theta$, glued from a stability condition of the family
\cite{BLMS} on $\AX$ and Ext-exceptional conditions on the
exceptional factors; in particular $\sigma_t$ is sectorial for
$t > t_0$, by Lemma~\ref{lem:sectoriality}. The entry clock is set
by the slowest separation, $t_0^{-1} \sim \min_i
|\operatorname{Im}(u_i e^{-i\theta})|$ --- by
Lemma~\ref{lem:centroid} these are exactly the gaps between $\AX$
and its two neighboring exceptional sectors.
\item \textup{(Walls)} Along the path, the boundary of the selection
region is contained in the failure locus of finitely many gluing
inequalities --- one for each forward morphism complex between
consecutive sectors, with thresholds bounded by the cohomological
amplitudes of those complexes --- and each entry wall is the strict
semistabilization of binding objects whose classes have the form
$v + w$ with $v \in \Mukai$, $v^2 \geq -2$, and $w$ in the lattice
spanned by the $[\mathcal{O}(i)]$; on the wall their
Jordan--H\"older factors are sectorial.
\item \textup{(Protection)} For $t > t_0$ there are no binding
objects; the counting invariants attached to Mukai vectors of $\AX$
are independent of $t$ by the CY2 invariance of counting invariants
\cite{Toda}; and Conjecture~8.1(i) of \cite{R} holds at $\sigma_t$:
no stable BPS trajectory connects the exceptional sectors to the $K3$
sector. All three clauses are mutation-robust: in particular they
transfer to the Kuznetsov ordering \eqref{eq:kuzSOD} itself along the
twisted-charge paths of \cite[Thm.~4.7]{KRZ}.
\end{enumerate}
\end{conjecture}

\begin{remark}[Proof architecture after \cite{KRZ}, and what remains
open]\label{rem:proof-structure}
The models, combined with a close reading of \cite{KRZ}, supply the
architecture and most of the quantitative inputs:

\begin{enumerate}[label=\alph*.]
\item \emph{Clocks and offsets:} the phase slopes of the exceptional
sectors are $-\operatorname{Im}(u_i e^{-i\theta})/\pi$ against slope
zero for $\AX$, by \eqref{eq:phaseslope}; the $O(1)$ offsets are the
explicit constants of the asymptotic expansions
\cite[(3.9), (3.14)]{KRZ}, established there through the
Sanda--Shamoto mutation-system estimates \cite{SS} and a family
version of Watson's lemma \cite[App.~A]{KRZ}. On the rank-$24$ block
those expansions control the classes with nonvanishing leading
coefficient, which suffices for the numerical central charge.
\item \emph{Finitely many inequalities:} the
backward complexes
$\RHom(\mathcal{O}(i), \AX)$ vanish by semiorthogonality. The forward
complexes are controlled by Serre duality, with \\
$\RHom(A, \mathcal{O}(i)) \cong
\RHom(\mathcal{O}(i-3), A)^\vee[-4]$.  They are also concentrated in a
bounded range of degrees, so from the results of
Remark~\ref{rem:koszul} the Collins--Polishchuk condition reduces
to finitely many phase inequalities with thresholds bounded by the
amplitude.  Also, the infinite enumeration of binding-class walls is
dominated by them (nonemptiness of the relevant Bridgeland moduli for
$v^2 \geq -2$ following the family theory of \cite{BLMNPS}).
\item \emph{Permanence:} sectoriality past entry is
Lemma~\ref{lem:sectoriality}; the chamber-extent statement requires,
in place of the vanishing backward $\Ext^1$ of the models, the
analysis of the nonvanishing groups
$\Ext^\bullet(A, \mathcal{O}(i))$, whose extensions however place the
\emph{exceptional} (higher-phase) object as subobject and are
therefore destabilized for the same reason as in
Lemma~\ref{lem:P1-extent}.
\item \emph{Phase control:} The paths of \cite[Thms.~4.6, 4.7]{KRZ} are glued at every time, so phase control of all ambient semistable objects holds along them by construction. 
\end{enumerate}

What remains splits into three parts:

\begin{enumerate}[label=\alph*.]
\item \emph{The wall layer} (this paper's program): both endpoints
of the conjecture exist in \cite{KRZ} --- a geometric glued
condition \cite[Thm.~2.72]{KRZ} and a quasi-convergent sectorial
tail \cite[Thm.~4.7]{KRZ}, connected by a single glued path whose
exceptional heart tilts through discrete shifts
\cite[Rem.~2.73]{KRZ}; those tilt times are precisely the walls of
clause (2), crossed but not analyzed in \cite{KRZ}, and identifying
the binding-object deaths and their Jordan--H\"older factors at each
is the outstanding content of the conjecture. (Quasi-convergence of
the Theorem~4.7 path is asserted with verification omitted in
\cite{KRZ}; the wall program does not rely on it.)
\item \emph{The
ordering layer:} the natural quantum charge produces the
interior-$\AX$ mutation $\mathbb{D}_\theta$
(Lemma~\ref{lem:centroid}), while the geometric endpoint is
currently reached only through a twisted fundamental solution;
reconciling the two is \cite[Conj.~2(C)]{KRZ} for cubics, and by the
centroid obstruction it cannot be achieved by a choice of ray.
\item \emph{The deformation layer:} the missing analytic technology is an
isomonodromic deformation of the cubic's quantum connection over a
base with braid-group fundamental group \cite[Question~4.8]{KRZ};
this lies beyond both Dubrovin's semisimple theory and the
coalescence theory of \cite{CDG}, which presupposes a semisimple
ring. Since the cubic's zero-eigenvalue block is not semisimple,
the natural framework is the $F$-bundle formalism of
\cite{KKPY,HYZZ}.
\end{enumerate}
\end{remark}

\begin{remark}[Exclusions]\label{rem:exclusions}
Two loci are excluded and merit comment. Cubics containing a plane
change the exceptional structure through the associated quadric
fibration and are excluded already in \cite{KRZ}. Separately, the
resonant model's rank-drop phenomenon (Lemma~\ref{lem:swap}(1):
$Z(\delta) \equiv 0$ at a symmetric point) suggests that cubics with
extra symmetry may admit sublattices of $\Mukai$ on which the quantum
central charge locks or degenerates; such loci would require separate
treatment in the support-property analysis of
Conjecture~\ref{conj:chamber}(1), and we flag them rather than
attempt a classification here.
\end{remark}

\section{The coupled case: $A_2$, the deformed cubic oscillator, and
Painlev\'e I}\label{sec:A2}

The Fano models of \S\S\ref{sec:P1}--\ref{sec:epsilon} calibrate the
chamber layer of the program; we now develop the complementary test
bed, which calibrates the analytic layer. It is the minimal
\emph{coupled} two-class structure: the $A_2$ quiver category, whose
Riemann--Hilbert problem is solved by the monodromy of the deformed
cubic oscillator \cite{BrMa} and whose tau function is Painlev\'e I.
Everything in this section is explicitly computable, allowing us to
test the analytic dictionary.

\subsection{The $A_2$ BPS structure and its chambers}
\label{subsec:A2bps}

Let $\mathcal{D}$ be the CY3 category of the Ginzburg dg-algebra of
the $A_2$ quiver $\bullet \to \bullet$ --- the finite-dimensional
derived category of the $3$-Calabi--Yau completion of the path
algebra \cite{Kel} --- with simple objects
$S_1, S_2$ satisfying $\Ext^1(S_1, S_2) = \CC$,
$\Ext^1(S_2, S_1) = 0$, charge lattice
$\Gamma = \ZZ\gamma_1 \oplus \ZZ\gamma_2$, $\gamma_i = [S_i]$, and
skew pairing $\langle \gamma_1, \gamma_2\rangle = 1$.

Recall from \cite[\S 2]{BrRH} that a \emph{BPS structure}
$(\Gamma, Z, \Omega)$ consists of a finite-rank free abelian group
$\Gamma$ with a skew-symmetric integral form
$\langle -, - \rangle$, a group homomorphism
$Z \colon \Gamma \to \CC$ (the \emph{central charge}), and a map
$\Omega \colon \Gamma \to \QQ$ (the \emph{BPS invariants})
satisfying the symmetry $\Omega(-\gamma) = \Omega(\gamma)$ and the
support property: for some norm $\|\cdot\|$ on
$\Gamma \otimes \RR$ there is $C > 0$ with
$|Z(\gamma)| > C\, \|\gamma\|$ whenever $\Omega(\gamma) \neq 0$.
The data of this subsection assembles into the family of $A_2$ BPS
structures over the unfolding space: the lattice and pairing above,
the period central charge \eqref{eq:periods}, and the
chamber-dependent invariants recorded below.

The quotient of
$\Stab(\mathcal{D})$ by spherical twists is identified with the
unfolding space of the $A_2$ singularity \cite{Sut,BS}: quadratic
differentials
\begin{equation}\label{eq:Q0}
  Q_0(x)\, dx^2 = (x^3 + a x + b)\, dx^2, \qquad (a,b) \in \CC^2,
\end{equation}
with central charges the periods of the spectral elliptic curve
$E_{a,b}\colon y^2 = Q_0(x)$,
\begin{equation}\label{eq:periods}
  Z(\gamma_i) = \oint_{\gamma_i} \sqrt{Q_0(x)}\, dx,
\end{equation}
where $\gamma_1, \gamma_2$ are the vanishing cycles at the two
branches of the discriminant locus $\{\Delta = 0\}$, where we
write $\Delta = 4a^3 + 27b^2$ in the normalisation of \cite{BrMa}.

The
$\CC^\times$-weights are fixed by $\mathrm{wt}(x) = 1$: then
$\mathrm{wt}(a) = 2$, $\mathrm{wt}(b) = 3$,
$\mathrm{wt}(Z) = 5/2$, and --- from the oscillator
\eqref{eq:oscillator} below --- $\mathrm{wt}(\hb) = 5/2 =
\mathrm{wt}(Z)$, so that $Z/\hb$ is scale-invariant.

The extension $0 \to S_2 \to E \to S_1 \to 0$, of class
$\gamma_1 + \gamma_2$, is stable precisely when
$\varphi(S_2) < \varphi(S_1)$. There are two chambers: a
\emph{two-state} chamber with stables $\{S_1, S_2\}$ and BPS
invariants $\Omega(\pm\gamma_1) = \Omega(\pm\gamma_2) = 1$, and a
\emph{three-state} chamber with additionally
$\Omega(\pm(\gamma_1{+}\gamma_2)) = 1$, separated by the wall of
marginal stability $\{Z(\gamma_2)/Z(\gamma_1) \in \RR_{>0}\}$, across
which the Kontsevich--Soibelman \emph{pentagon} identity holds
\cite{KS}:
\begin{equation}\label{eq:pentagon}
  \mathbb{U}_{\gamma_1}\,\mathbb{U}_{\gamma_2}
  \;=\;
  \mathbb{U}_{\gamma_2}\,\mathbb{U}_{\gamma_1+\gamma_2}\,
  \mathbb{U}_{\gamma_1}.
\end{equation}
The binding object $E$ is the exact CY3 analogue of the torsion
sheaves of \S\ref{sec:P1} --- a stable state existing on one side of
a phase-alignment wall --- with one decisive structural difference
recorded in \S\ref{subsec:coupled-uncoupled}.

We record one explicit landmark, both as a benchmark for the numerics
of \S\ref{subsec:det} and to caution against a tempting but false
analogy. On the $\ZZ_2$-symmetric slice $b = 0$, $a < 0$ real, the
turning points are $\{-s, 0, s\}$ with $s = \sqrt{-a}$, and a direct
computation of \eqref{eq:periods} gives: $Z(\gamma_1)$ (the cycle
over $[-s,0]$) is real, $Z(\gamma_2)$ (over $[0,s]$) is purely
imaginary, with common modulus
\begin{equation}\label{eq:beta-period}
  |Z(\gamma_1)| = |Z(\gamma_2)|
  = 2\int_0^{s} \sqrt{x\,(s^2 - x^2)}\; dx
  = s^{5/2}\, B\!\left(\tfrac34, \tfrac32\right)
  = s^{5/2}\,
  \frac{\Gamma(\tfrac34)\,\Gamma(\tfrac32)}{\Gamma(\tfrac{9}{4})}
  \approx 0.958512\; s^{5/2}.
\end{equation}
Thus the symmetric slice is a \emph{mass} resonance with
\emph{orthogonal} phases, lying in the interior of a chamber ---
equidistant from the wall, not on it. In particular the deformation
parameter $b$ is \emph{not} the analogue of the resonance-breaking
$\varepsilon$ of \S\ref{sec:epsilon}; the correct relation between
the two models is the coupled/uncoupled dichotomy of
\S\ref{subsec:coupled-uncoupled} and
Remark~\ref{rem:coupled-preview}.

\subsection{Exact WKB and the pentagon}\label{subsec:wkb}

The Schr\"odinger operator
\begin{equation}\label{eq:oscillator-bare}
  \hb^2\, y''(x) = Q_0(x)\, y(x)
\end{equation}
has five Stokes sectors at infinity (solutions $\sim
\exp(\pm\tfrac{2}{5}x^{5/2}/\hb)$), one subdominant solution per
sector, and a two-dimensional manifold of Stokes data --- matching
$\operatorname{rk}\Gamma = 2$. The Voros symbols $V_\gamma(\hb)$ --- exponentials of the
all-orders WKB periods along the cycles $\gamma$ --- carry
asymptotics with leading behavior $e^{-Z(\gamma)/\hb}$; their Stokes
automorphisms across saddle-connection phases are governed by the
Delabaere--Dillinger--Pham formula \cite{DDP}.

The DDP
identity for the cubic potential is the pentagon \eqref{eq:pentagon}: the analytic wall-crossing of the oscillator
computes the categorical wall-crossing of $\mathcal{D}$. This is the
oldest fully explicit instance of the identification of Stokes data
with Donaldson--Thomas data and the reason $A_2$ is the right model
on the analytic side.

For \S\ref{subsec:det} the classical fact \cite{Sib}
that the Stokes multipliers of such operators are themselves values
of spectral determinants of rotated problems.

\subsection{The Bridgeland--Masoero deformation}\label{subsec:BM}

The Riemann--Hilbert problem of \cite{BrRH} attached to the $A_2$ BPS
structure asks for functions $x_\gamma(\hb)$ on half-planes, jumping
across each BPS ray $\ell_\gamma = \RR_{>0}Z(\gamma)$ by the
birational transformation
\begin{equation}\label{eq:jump}
  x_\beta \;\longmapsto\; x_\beta\,
  (1 - x_\gamma)^{\Omega(\gamma)\langle \gamma, \beta\rangle},
\end{equation}
with prescribed asymptotics
$x_\gamma(\hb)\, e^{Z(\gamma)/\hb} \to \xi_\gamma$ as $\hb \to 0$.

When the pairing vanishes on the support of $\Omega$ (the
\emph{uncoupled} case, e.g.\ the resolved conifold), the problem
factorizes and is solved by Barnes-type special functions
\cite{BrRH,BrCon}. The $A_2$ structure, with
$\langle\gamma_1,\gamma_2\rangle = 1$, is instead the minimal \emph{coupled}
case, and requires a new transcendental: Painlev\'e I
\cite{BrMa}.

The mechanism of \cite{BrMa} is the deformed cubic oscillator
\begin{equation}\label{eq:oscillator}
  y''(x) \;=\; Q(x,\hb)\, y(x), \qquad
  Q(x,\hb) \;=\; \hb^{-2} Q_0(x) + \hb^{-1} Q_1(x) + Q_2(x),
\end{equation}
with potentials \cite[Eq.~(2)]{BrMa}
\begin{equation}\label{eq:BMpotentials}
  Q_0(x) = x^3 + a x + b, \qquad
  Q_1(x) = \frac{p}{x-q} + r, \qquad
  Q_2(x) = \frac{3}{4(x-q)^2} + \frac{r}{2p(x-q)} + \frac{r^2}{4p^2},
\end{equation}
where the coefficient $\tfrac34$ fixes the indicial exponents at
$x = q$ to $\{-\tfrac12, \tfrac32\}$ and the singularity is required
to be \emph{apparent} (trivial local monodromy); $r$ is an auxiliary
parameter. The apparency condition admits a short exact derivation
which we record, since it exhibits the conceptual point:

\begin{lemma}\label{lem:apparent}
For the potentials \eqref{eq:BMpotentials}, the point $x = q$ is an
apparent singularity of \eqref{eq:oscillator} if and only if
\[
  p^2 \;=\; Q_0(q) \;=\; q^3 + a q + b,
\]
identically in $\hb$ and in $r$: the deformation datum $(q,p)$ is a
point of the spectral elliptic curve $E_{a,b}$ whose periods are the
central charges \eqref{eq:periods}.
\end{lemma}

\begin{proof}
Multiply \eqref{eq:oscillator} by $\hb^2$, set $t = x - q$ and
$y = t^{-1/2} u(t)$ with $u = \sum_{k \geq 0} c_k t^k$, $c_0 = 1$.
The double pole cancels, and with pole coefficient
$\hb p + \hb^2 \tfrac{r}{2p}$ at $t^{-1}$ and regular part
$Q_0(q) + \hb r + \hb^2 \tfrac{r^2}{4p^2} + O(t)$, the recursion
reads $\hb^2 k(k-2) c_k = \bigl(\hb p + \hb^2\tfrac{r}{2p}\bigr)
c_{k-1} + (\text{regular})\cdot(\text{lower } c)$. At $k=1$:
$c_1 = -\bigl(\tfrac{p}{\hb} + \tfrac{r}{2p}\bigr)$. At $k = 2$ ---
the resonance of the exponents, where the logarithm would appear ---
the left side vanishes identically and the right side must too:
\begin{align*}
  0 &= \Bigl(\hb p + \hb^2\tfrac{r}{2p}\Bigr)
  \Bigl(-\tfrac{p}{\hb} - \tfrac{r}{2p}\Bigr)
  + Q_0(q) + \hb r + \hb^2\tfrac{r^2}{4p^2} \\
  &= -\Bigl(p^2 + \hb r + \hb^2\tfrac{r^2}{4p^2}\Bigr)
  + Q_0(q) + \hb r + \hb^2\tfrac{r^2}{4p^2}
  \;=\; Q_0(q) - p^2 :
\end{align*}
all $r$- and $\hb$-dependent terms cancel exactly.
\end{proof}

The theorem of \cite{BrMa} we build on is: the framed Stokes data
$x_1(\hb), x_2(\hb)$ of \eqref{eq:oscillator} solve the $A_2$
Riemann--Hilbert problem for the BPS structure with central charge
$Z_{a,b}$; variation of $(a,b)$ at fixed monodromy is an
isomonodromic flow equivalent, after the weight-$(2,3)$ scaling of
\S\ref{subsec:A2bps}, to \emph{Painlev\'e I}, with $q$ as the
Painlev\'e transcendent along the flow: explicitly, in the
normalization of \cite{BrMa} with the auxiliary parameter $r = 0$,
\begin{equation}\label{eq:PIflow}
  \dot a = 1, \qquad \dot b = -q, \qquad
  \dot q = -\frac{2p}{\hb}, \qquad
  \dot p = -\frac{3q^2 + a}{\hb},
\end{equation}
which preserves the apparency locus $p^2 = Q_0(q)$ of
Lemma~\ref{lem:apparent} and reduces to
$\hb^2\, \ddot q = 6 q^2 + 2a$, Painlev\'e~I with $a$ as the
time; and the tau function of the
family equals the Painlev\'e I tau function up to explicit elementary
factors.

Zeros of $\tau$ then correspond to poles of the transcendent ---
the apparent singularity escaping to infinity, the operator
degenerating to \eqref{eq:oscillator-bare} --- by the
pole/oscillator correspondence of Masoero \cite{Mas}.

\subsection{The dictionary}\label{subsec:dictionary}

We can summarize the categorical and analytic identifications in the two pictures as follows:

\begin{center}
\footnotesize
\begin{tabular}{|l|l|}
\hline
categorical side & $A_2$ / oscillator side \\
\hline
class $\gamma \in \Gamma$ & cycle on $E_{a,b}\colon y^2 = x^3+ax+b$ \\
central charge $Z(\gamma)$ & WKB period
  $\oint_\gamma \sqrt{Q_0}\,dx$ \\
stable object & saddle trajectory of $Q_0\,dx^2$ \\
binding object (\S\ref{sec:P1}: $\mathcal{O}_x$;
  \S\ref{sec:product}: $\mathcal{O}_D$) &
  the extension $E$, class $\gamma_1+\gamma_2$ \\
chamber wall / entry wall & saddle connection,
  $Z_2/Z_1 \in \RR_{>0}$ \\
KS jump; pentagon \eqref{eq:pentagon} & DDP Stokes automorphism \\
DT invariants $\Omega = 1$ & simple jumps of framed Stokes data \\
stability direction & $(a,b)$: isomonodromic; flow $=$ PI \\
spectral direction & $\hb$: jumps across BPS rays \eqref{eq:jump} \\
RH solution $x_\gamma(\hb)$ & Stokes multipliers \\
$\tau$ & $\tau_{\mathrm{PI}}$, up to elementary factors \\
zeros of $\tau$ & poles of PI; apparent singularity at $\infty$
  \cite{Mas} \\
``partition function of a D-brane sector'' \cite[\S 8]{R} &
  long-distance expansions of $\tau_{\mathrm{PI}}$ via \\
 & \quad $H_0 = (A_1,A_2)$ AD partition functions \\
\hline
\end{tabular}
\end{center}

\smallskip

The last row deserves emphasis as it closes the loop with the physical
reading proposed in \cite[\S 8]{R}. The analytic index realized as
the partition function of a D-brane sector is an example
where the identification of the tau function with gauge-theory
partition functions is available at the level of established
expansions in the Painlev\'e--gauge correspondence \cite{BLMST}.

\subsection{Decoupled directions}\label{subsec:decoupled}

The $A_2$ setting cleanly separates two directions that the Fano
setting conflates. Here the stability direction $(a,b)$ and the
spectral direction $\hb$ are independent: the Riemann--Hilbert
problem lives at fixed $(a,b)$ with variable $\hb$, and
$(a,b)$-variation is isomonodromic --- which is precisely \emph{why}
a tau function exists as a function on the stability space. In the
Fano models the single parameter $t$ does double duty: through the
$F$-bundle, the base of the quantum connection is simultaneously the
spectral plane and, via the lifted path, a direction in $\Stab$. The
correct comparison is therefore two-layered, matching
Remark~\ref{rem:layers}: our entry walls in $t$ correspond to the
$A_2$ chamber walls in $(a,b)$ (chamber layer), while the Stokes rays
of the quantum connection at fixed $t$ correspond to the $\hb$-rays
at fixed $(a,b)$ (spectral layer). Any analytic formulation for the
Fano case must take care to properly address these two roles.

\subsection{The determinant program at $A_2$}\label{subsec:det}

The proposal of \cite[\S 8]{R}, in the sharpened form suggested by
the present framework --- recover $\log\tau$ from a zeta-regularized
determinant of the operator family, with the JLO Chern character
\cite{JLO,GS} acting fiberwise through the superconnection formalism
--- decomposes at $A_2$ into three steps, of which the first is
classical:

\begin{enumerate}
\item \emph{Stokes data are determinants (classical).} By \cite{Sib},
each Stokes multiplier of \eqref{eq:oscillator-bare} and
\eqref{eq:oscillator} is a value of a spectral determinant of a
rotated problem. The Riemann--Hilbert solution of \cite{BrMa} is
therefore already assembled from determinants; what is missing is
only the assembly at the level of $\tau$.
\item \emph{Assembly (the open content).} Express the tau function of
the family --- defined through the $(a,b)$-variation of the
$x_\gamma$ --- as a zeta-regularized determinant of the family
\eqref{eq:oscillator} over $(a,b)$, with the Quillen metric supplying
the determinant line and its connection, whose curvature should
reproduce the symplectic structure on the $(a,b)$-space underlying
the Joyce structure \cite{BrTau}. Two sharp consistency checks are
forced: (i) the small-$\hb$ asymptotics of $\log\det$ must reproduce
the periods \eqref{eq:periods} (WKB of the determinant $=$ Agmon
actions) --- \emph{verified at leading order in Appendix~\ref{app:det},
by two independent routes, and continued to all perturbative
orders in \S\ref{sec:detline}}; (ii) the zero
locus of the assembled determinant must be the pole locus of
\cite{Mas}.
\item \emph{Comparison.} Independent representations of
$\tau_{\mathrm{PI}}$ from the Painlev\'e--gauge correspondence
\cite{BLMST} supply a second target, and a numerical verification at
a single generic $(a,b)$ is entirely feasible: Stokes data of
\eqref{eq:oscillator} by standard collocation, the pentagon
\eqref{eq:pentagon} as the internal check, and the explicit period
\eqref{eq:beta-period} at $b = 0$ as the benchmark for the
asymptotic check (i).
\end{enumerate}

If step (2) succeeds, the JLO formulation of \cite{R} acquires its
first theorem: the cyclic-cohomological Chern character data of an
operator family assembles to a Donaldson--Thomas tau function in the
minimal coupled case.

\subsection{Coupled versus uncoupled: the two minimal models}
\label{subsec:coupled-uncoupled}

We can now state precisely how the two sides considered in this work
relate, and why both are needed.

First, \emph{no glued region exists on the $A_2$ side, and its wall
is two-sided.} By Lemma~\ref{lem:CY-indec} (whose Serre-duality
argument requires only $S \cong [n]$, and whose $\HH^0$-hypothesis
can here be bypassed via $\Ext^1(S_1,S_2) \neq 0$), the category
$\mathcal{D}$ admits no nontrivial semiorthogonal decomposition:
hence no glued region, no selection region, no terminality. The
pentagon wall is crossed freely in both directions, and the
invariance statement is the wall-crossing formula
\eqref{eq:pentagon}, not a chamber theorem. The $A_2$ case therefore
calibrates the \emph{analytic} dictionary only; the \emph{chamber}
dictionary is calibrated by the Fano models of
\S\S\ref{sec:P1}--\ref{sec:epsilon}, where semiorthogonal
decompositions exist and the flow terminates in a sectorial regime.

Second, the correct classification of the two-class local models is
by the pairing, as anticipated in
Remark~\ref{rem:coupled-preview}. The middle block of the resonant
product is the minimal \emph{uncoupled} degeneracy:
$\langle\cdot,\cdot\rangle = 0$ between the middle classes, no wall,
no binding morphism, region picture phenomena only, exact
phase-locking --- and, on the Riemann--Hilbert side, Barnes-level
solvability. The $A_2$ structure is the minimal \emph{coupled} pair:
$\langle\gamma_1,\gamma_2\rangle = 1$, a genuine two-sided wall, the
pentagon, and Painlev\'e-level transcendence. Together the two
exhaust the local models of a two-class interaction, and they sit on
opposite sides of the same dichotomy that separates the solvable from
the transcendental Riemann--Hilbert problems. (The computation
\eqref{eq:beta-period} shows the $\ZZ_2$-symmetric slice of the
$A_2$ family is a mass resonance with orthogonal phases in the
interior of a chamber --- so the deformation $b$ is \emph{not} an
analogue of the $\varepsilon$ of \S\ref{sec:epsilon}; the analogy
runs through the pairing, not through the unfolding parameters.)

Lastly, \emph{massless loci sit at finite distance.} The
conifold-type points $\Delta = 0$, where some $Z(\gamma) \to 0$ and a
stable object goes massless, are interior degenerations of the
$(a,b)$-space reached at finite parameter values, in contrast to
the Fano paths, along which the block goes massless only in the
$t \to \infty$ limit, on the augmented boundary
(Remark~\ref{rem:massless}). The heat-kernel counterpart, where the
well's localization weight collapses, can be studied
at finite parameters on the $A_2$ side, which is
a practical advantage for the
analytic side of the investigation.

\section{The determinant line at $A_2$: corrections, flatness, and
curvature}\label{sec:detline}

The consistency requirements of the determinant program of
\S\ref{subsec:det}, step (2), begin with the identification of the
small-$\hb$ growth of the spectral determinant with the period data.
This is carried out at leading order in Appendix~\ref{app:det}, by
two independent routes --- zero-counting and the analytic
continuation of the periods --- with the result
\begin{equation}\label{eq:det-dict}
  \log D(a,b;\hb) \;=\; \frac{1}{2\hb}
  \sum_{\gamma \in V(D)} Z(\gamma)(a,b) \;+\; O(1),
\end{equation}
where $V(D)$ is the set of vanishing cycles adjacent to the
determinant's boundary sectors, continued along the parameter ray;
the a priori $O(\log \hb^{-1})$ term is shown there to vanish, and
the symplectic constant of the family geometry,
$dZ(\gamma_1) \wedge dZ(\gamma_2) = \pm\, 2\pi i\, da \wedge db$
\eqref{eq:app-symp}, is computed by residue. This section continues
the calculation into the perturbative regime: the $\hb$-corrections
to \eqref{eq:det-dict}, obtained in closed form
(\S\ref{subsec:det-qp}); their organisation into a deformation that
is flat to all orders, with the exact curvature concentrated on the
zero divisor (\S\ref{subsec:det-flat}); and the reduction of the
elementary-factor ambiguities of the two normalisations to a single
constant (\S\ref{subsec:det-elem}).

\subsection{Quantum periods and the first $\hb$-correction}\label{subsec:det-qp}

With $y = \exp(\hb^{-1}\!\int P\,dx)$ and
$P = \sum_k \hb^k s_k$, the Riccati recursion
$Q_0 = \hb P' + P^2$ gives $s_0 = \sqrt{Q_0}$,
$s_1 = -Q_0'/4Q_0$, and
\[
  s_2 \;=\; \frac{Q_0''}{8\,Q_0^{3/2}}
  \;-\; \frac{5\,Q_0'^2}{32\,Q_0^{5/2}}.
\]
The odd-order terms are exact, so the periods of $P$ define the
quantum periods
$Z(\gamma, \hb) = Z_0(\gamma) + \hb^2 Z_2(\gamma) + O(\hb^4)$ with
$Z_2(\gamma) = \oint_\gamma s_2\, dx$; this is the undeformed
specialisation of the Voros data of \cite[App.]{BrMa}. On $E_{a,b}$
the form $s_2\,dx$ is anti-invariant under the elliptic involution
and of the second kind: at each ramification point (including
$\infty$, local coordinate $x^{-1/2}$) anti-invariance forces even
powers in the local expansion, so all residues vanish.
Griffiths--Dwork reduction --- write the numerator as
$u\,Q_0 + v\,Q_0'$ via the extended gcd and integrate by parts at
each pole order --- terminates in the basis
$A(\gamma) = \oint_\gamma dx/y$,
$B(\gamma) = \oint_\gamma x\,dx/y$, and yields, exactly,
\begin{equation}\label{eq:det-Z2red}
  Z_2(\gamma) \;=\; \frac{a^2}{4\Delta}\, A(\gamma)
  \;-\; \frac{9b}{8\Delta}\, B(\gamma).
\end{equation}
Since $A = 2\,\partial_b Z_0$ and $B = 2\,\partial_a Z_0$, this is a
$\gamma$-\emph{independent} first-order operator, with the closed
form
\begin{equation}\label{eq:det-Z2}
  Z_2 \;=\; \frac{1}{24}\,\bigl\{\log\Delta,\; Z_0\bigr\},
  \qquad
  \{f, g\} \;:=\; \frac{\partial f}{\partial a}
  \frac{\partial g}{\partial b}
  - \frac{\partial f}{\partial b}
  \frac{\partial g}{\partial a}:
\end{equation}
the first quantum correction of the periods is the Hamiltonian
vector field of $\tfrac{1}{24}\log\Delta$ --- for the symplectic
structure normalised by \eqref{eq:app-symp} --- acting on the
classical periods. (We verified \eqref{eq:det-Z2red} by direct
contour integration at generic $(a,b)$.) The dictionary \eqref{eq:det-dict} extends to
\begin{equation}\label{eq:det-dict2}
  \log D(a,b;\hb) \;=\; \frac{1}{2\hb}
  \sum_{\gamma \in V(D)} Z(\gamma, \hb)(a,b) \;+\; O(1)
  \;=\; \frac{1}{2\hb} \sum_{V(D)} Z_0 \;+\; O(1)
  \;+\; \frac{\hb}{2} \sum_{V(D)} Z_2 \;+\; O(\hb^3).
\end{equation}

On the symmetric slice $b = 0$, \eqref{eq:det-Z2red} collapses to
$Z_2(\gamma) = A(\gamma)/16a$. On the eigenvalue axis this gives
$|Z_2(\gamma_2)| = B(\tfrac14, \tfrac12)\,|a|^{-5/4}/16$, which is
the $k = 1$ term of the all-orders refinement of the counting
function,
\begin{equation}\label{eq:det-counting}
  N(\lambda) \;=\; \frac{\beta}{2\pi\hb}\,\lambda^{5/4}
  \;+\; \frac12
  \;+\; \sum_{k \geq 1} c_k\, \hb^{2k-1} \lambda^{5/4 - 5k/2},
\end{equation}
Continued to $a > 0$ along the path of Route~$\beta$ of
Appendix~\ref{app:det},
$Z_2(\gamma_\pm) = e^{\pm i\pi/4}\,
(B(\tfrac14,\tfrac12)/16)\, a^{-5/4}$, and the two routes agree
again at order $\hb$: on the cycle side,
$\tfrac\hb2 (Z_2^+ + Z_2^-) =
\hb \cos\tfrac\pi4 \cdot (B(\tfrac14,\tfrac12)/16)\, a^{-5/4}$; on
the counting side, the correction term of \eqref{eq:det-counting}
has exponent $\rho - \tfrac52 = -\tfrac54$ and Lindel\"of factor
$\pi/\sin(-\tfrac{5\pi}{4}) = \pi\sqrt2$, giving the same
\begin{equation}\label{eq:det-hbcorr}
  \frac{\hb}{2}\sum_{V(D)} Z_2
  \;=\; \frac{B(\tfrac14, \tfrac12)}{16\sqrt2}\cdot
  \frac{\hb}{a^{5/4}}
  \;\approx\; 0.231759\;\frac{\hb}{a^{5/4}},
  \qquad a \to +\infty, \; b = 0.
\end{equation}
The sign flip of $\sin\pi\rho$ under $\rho \mapsto \rho - \tfrac52$
is exactly the rotation $e^{3i\pi/4} \mapsto e^{i\pi/4}$ of the
continued cycle phases: the order-$\hb$ agreement is the same
monodromy statement as the leading one, one weight lower.
(Both displays take the visible cycle set of $D$ to be the conjugate
pair $\{\gamma_+, \gamma_-\}$ with weight $\tfrac12$; the modulus
statements, and the coefficient of \eqref{eq:det-hbcorr}, are
confirmed numerically in \S\ref{subsec:det-numerics}, where the
phase analysis further identifies the equivalent single-cycle form
\eqref{eq:det-singlecycle}.)

\begin{remark}\label{rem:det-joyce}
The proof of \cite[Prop.~4.4]{BrMa} computes the linear term of the
Joyce function at the zero-section:
$(2\pi i)^{-1}\partial J/\partial\theta_a|_0 = 2a^2\!/4\Delta$ and
$(2\pi i)^{-1}\partial J/\partial\theta_b|_0 = 9b/4\Delta$. Writing
\eqref{eq:det-Z2red} through the periods and quasi-periods,
$Z_2(\gamma_i) = \bigl(\tfrac{a^2}{2}\omega_i +
\tfrac{9b}{4}\eta_i\bigr)/\Delta$, one finds the same coefficient
pair, and in the conventions of \cite[(49)--(50)]{BrMa}, with
$\langle\gamma_1, \gamma_2\rangle = 1$,
\[
  Z_2(\gamma_i) \;=\; \sum_j \langle \gamma_i, \gamma_j \rangle\,
  \frac{\partial J}{\partial \theta_j}\Big|_{\theta = 0}.
\]
The first $\hb$-correction of the determinant's growth is the
Hamiltonian generator of the Joyce function along the zero-section:
the determinant sees the Joyce structure one order in $\hb$ before
any curvature appears.
\end{remark}

\subsection{All orders: perturbative flatness and the curvature
current}\label{subsec:det-flat}

\begin{proposition}\label{prop:det-allorders}
For every $k \geq 1$ there are weighted-homogeneous
$\alpha_k, \beta_k \in \mathbb{Q}[a,b][\Delta^{-1}]$, of weights $3 - 5k$
and $2 - 5k$ and independent of $\gamma$, such that
$Z_{2k}(\gamma) = \alpha_k A(\gamma) + \beta_k B(\gamma)$ for every
$\gamma \in H_1(E_{a,b}, \ZZ)$. Equivalently
$Z_{2k} = \mathsf{D}_k Z_0$ with
$\mathsf{D}_k = 2\beta_k \partial_a + 2\alpha_k \partial_b$ a
first-order operator regular away from $\Delta = 0$.
\end{proposition}

\begin{proof}[Proof sketch]
Each $s_{2k}\,dx$ is anti-invariant under the elliptic involution,
so the parity argument above kills every residue and the form is of
the second kind; the Griffiths--Dwork reduction then terminates in
the span of $dx/y$ and $x\,dx/y$ modulo exact forms, with
coefficients rational in $(a,b)$ and poles only on $\Delta$ by the
extended-gcd construction; $\gamma$-independence holds because the
reduction is an identity of forms. (Residue-freeness for the
deformed corrections $\alpha_k$ of the full Voros data is
\cite[Cor.~A.12]{BrMa}.)
\end{proof}

\begin{corollary}[perturbative flatness]\label{cor:det-flat}
On any simply connected region with $\Delta \neq 0$ on which $V(D)$
is locally constant, every coefficient of the asymptotic expansion
\[
  d \log D \;\sim\; \sum_{k \geq 0} \frac{\hb^{2k-1}}{2}
  \sum_{\gamma \in V(D)} dZ_{2k}(\gamma) \;+\; d\bigl(O(1)\bigr)
\]
is an exact form. The curvature of the determinant line vanishes to
all orders in $\hb$ away from $\Delta$ and the parameter walls; the
$\hb$-corrections deform the flat Agmon connection of
\eqref{eq:det-dict} only by exact forms with poles on the
discriminant.
\end{corollary}

The exact statement sits behind the expansion. At fixed $\hb$, by
Poincar\'e--Lelong,
\begin{equation}\label{eq:det-PL}
  \frac{i}{\pi}\,\partial\bar\partial \log |D| \;=\; [\,\mathrm{div}\, D\,]
\end{equation}
as currents on $\CC^2_{(a,b)}$: the curvature of the determinant
line, computed from its canonical holomorphic section, is
$2\pi i\,[\mathrm{div}\, D]$ exactly --- a current supported on the
zero divisor with no smooth part, consistent with
Corollary~\ref{cor:det-flat}. Granting check (ii) of
\S\ref{subsec:det} --- the zero locus is the pole locus of
Painlev\'e~I \cite{Mas}, \ie\ the zero divisor of $\tau$ --- the
curvature current \emph{is} $2\pi i$ times the $\tau$-divisor:
check~(ii) and the curvature half of the consistency requirements
are one statement, and proving either proves both. Semiclassically
the divisor is the union of sheets
$\{Z(\gamma_{\mathrm{act}}, \hb) \in 2\pi i \hb (\ZZ + \tfrac12)\}$
in each ray region; the sheets accumulate on the anti-Stokes wall
$\{\Re Z(\gamma_{\mathrm{act}}) = 0\}$ with transverse spacing
$2\pi\hb$ in $\Im Z$, so at the level of densities the curvature
class per unit parameter volume is governed by
$(2\pi\hb)^{-1} dZ_1 \wedge dZ_2 = \pm\, i\, \hb^{-1}\, da \wedge
db$ by \eqref{eq:app-symp}: the determinant line is asymptotically
a prequantum line bundle for the unfolding symplectic form at level
$\hb^{-1}$.

The remaining analytic content of the curvature investigation
is to transfer from the singular to the smooth picture.
The Chern connection of the Quillen metric (holomorphic-section
curvature plus the $\partial\bar\partial$ of the zeta anomaly) must
reproduce this class as a smooth form.
Corollary~\ref{cor:det-flat} shows that no perturbative term
obstructs this. The smooth Quillen curvature is then a purely
non-perturbative object, the semiclassical average of
\eqref{eq:det-PL}.

\subsection{The elementary factors}\label{subsec:det-elem}

Both normalisation ambiguities of the assembly step live in an
explicitly describable finite-dimensional space, and this can be
established before the assembled object is constructed.

By
\cite{Sib}, $D(a,b;\hb)$ is entire on $\CC^2$ of finite order in
each variable; the Lindel\"of computation of
Appendix~\ref{app:det} gives order $\tfrac54$ in $a$, and the
$b$-direction analogue order $\tfrac56$ (Sibuya's classical
orders).

\begin{lemma}\label{lem:det-ambiguity}
Let $F, G$ be entire functions on $\CC^2_{(a,b)}$ with the same zero
divisor, each of order $\leq \tfrac54$ in $a$ at fixed $b$ and of
order $\leq \tfrac56$ in $b$ at fixed $a$, locally uniformly in the
other variable. Then
\[
  F/G \;=\; \exp\bigl(c_0 + c_1 a\bigr),
  \qquad c_0, c_1 \in \CC.
\]
\end{lemma}

\begin{proof}
The quotient is entire and zero-free on the simply connected
$\CC^2$, hence of the form $e^g$ with $g$ entire. Fix $b$: by the
Hadamard minimum-modulus estimate the zero-free function
$e^{g(\cdot, b)}$ has order $\leq \tfrac54$ in $a$, so $g(\cdot, b)$
is a polynomial of degree $\leq \lfloor \tfrac54 \rfloor = 1$:
$g = c_0(b) + c_1(b)\,a$ with $c_0, c_1$ entire in $b$. Fixing $a$
at two values, the same argument in the $b$-direction with order
$\tfrac56 < 1$ forces degree $0$: both $c_i$ are constant.
\end{proof}

\emph{The zeta side.} Set $\lambda_n = -a_n > 0$ and
$\zeta(s) = \sum_n \lambda_n^{-s}$, convergent for
$\Re s > \tfrac54$. The exponent lattice
$\{\tfrac54 - \tfrac52 k\}_{k \geq 0} \cup \{0\}$ of the
all-orders counting function \eqref{eq:det-counting} contains no
positive integer. Hence $\zeta(s)$ continues regularly to $s = 1$
and $s = 0$, with $\zeta(0) = \tfrac12$ from the Maslov constant ---
the same half-integrality of weights responsible for
$\kappa_{\log} = 0$. The zeta-normalized determinant is then fixed
by the classical identity \cite{Vor}
\begin{equation}\label{eq:det-zetanorm}
  D_\zeta(a) \;=\; e^{-\zeta'(0) \,-\, a\,\zeta(1)}
  \prod_n \Bigl(1 + \frac{a}{\lambda_n}\Bigr) e^{-a/\lambda_n}
  \;=\; \exp\Bigl(-\partial_s\big|_{s=0}
  \sum_n (\lambda_n + a)^{-s}\Bigr),
\end{equation}
valid precisely because $\zeta$ is regular at $s = 0, 1$. By
Lemma~\ref{lem:det-ambiguity} the discrepancy between Sibuya's
normalisation and \eqref{eq:det-zetanorm} is exhausted by the pair
$(\zeta'(0), \zeta(1))$. Weighted covariance pins the
$\hb$-dependence: on the symmetric slice the eigenvalues scale as
$a_n(\hb) = \hb^{4/5} a_n(1)$, so
$\zeta_\hb(s) = \hb^{-4s/5}\zeta_1(s)$ and
\begin{equation}\label{eq:det-anomaly}
  -\zeta_\hb'(0) \;=\; c_0^{0} + \tfrac{2}{5}\log\hb,
  \qquad
  \zeta_\hb(1) \;=\; c_1^{0}\, \hb^{-4/5},
\end{equation}
with $c_0^0 = -\zeta_1'(0)$, $c_1^0 = \zeta_1(1)$ pure numbers; the
coefficient $\tfrac25 = \tfrac45\,\zeta(0)$ is the scaling anomaly
of the zeta normalisation. (For $b \neq 0$ the same covariance makes
the constants weight-zero functions of $b\,\hb^{-6/5}$, constant
along each weighted ray.) Note the two combinations
$\hb^{-4/5} a$ and $\log \hb$ are exactly the weight-zero monomials
available --- the ambiguity is as small as scaling allows.

\emph{The tau side.} In the Joyce-structure formulation the
$\tau$-function is fixed by choices of symplectic potentials, and
with the choices of \cite[\S 8]{BrTau} it coincides, up to a
multiplicative constant, with the extension of
$\tau_{\mathrm{PI}}$ as a function of monodromy constructed by
Lisovyy--Roussillon \cite{LR}. The residual elementary freedom is
the finite list of potential shifts of \cite[\S 6.3]{BrTau}; on the
sector of elementary $(a,b)$-dependence it is generated by the
cotangent-structure shift
$\tfrac12 i_E(\lambda) = \tfrac25\, ab$, entering the $\hb$-graded
formula with the prefactor $\hb^{-2}$ --- the unique weight-$5$
monomial --- together with overall constants. Thus the two
ambiguity spaces are
\[
  \text{zeta side: } \exp\bigl[c_0^0 + \tfrac25 \log\hb
  + c_1^0\, \hb^{-4/5} a \bigr],
  \qquad
  \text{tau side: } \exp\bigl[c_0' + c'\, \hb^{-2}\, ab\bigr],
\]
which intersect only in the constants.

\emph{The cancellation mechanism.} The five Stokes sectors are
permuted by $\mu_5$, acting on parameters by
$(a, b) \mapsto (\lambda^4 a, \lambda^6 b)$, $\lambda^5 = 1$, with
$\hb$ fixed (the integer weights $(4,6,5)$ of
\cite[Rmk.~3.5]{BrMa}).

\begin{proposition}\label{prop:det-elem}
In any $\mu_5$-equivariant assembly of the rotated zeta-determinants,
the $a$-linear normalisers of \eqref{eq:det-anomaly} cancel
identically, since $\sum_{\lambda \in \mu_5} \lambda^4 = 0$, and the
logarithmic anomalies add to the pure prefactor $\hb^{2}$, absorbed
by the weighted homogeneity used to normalise $\tau$. The
elementary-factor matching against the \cite[\S 8]{BrTau}
normalisation of $\tau_{\mathrm{PI}}$ therefore reduces to: (i) a
single numerical constant, built from $\zeta_1'(0)$; and (ii) the
selection of the symplectic potential
$\Theta_0 \in \{\Theta_0^{\mathrm{L}}, \Theta_0,
\Theta_0^{\mathrm{H}}\}$ of \cite[\S 6.3]{BrTau} --- since the zeta
ambiguity contains no $ab$-term, at most one member of the family is
consistent.
\end{proposition}

\begin{proof}
Immediate from Lemma~\ref{lem:det-ambiguity},
\eqref{eq:det-anomaly}, and the fifth-root-of-unity sum.
\end{proof}

Both residues are numerically accessible: (i) as the fitted $O(1)$
constant of the collocation; (ii) as the $ab$-coefficient of the
assembled log-determinant against $d\log\tau_{\mathrm{LR}}$. A
closed form for $\zeta_1(1)$ and $\zeta_1'(0)$ may be within reach
of the ODE/IM representation of the cubic spectral sums, but is not
needed for the matching.

\subsection{Numerical verification on the symmetric slice}
\label{subsec:det-numerics}

The Sibuya solution $y_0(x, a)$ is computed by exponent-stripped
integration along rays in $S_0$, seeded at large radius with the odd
Voros data through $s_6$ and the exact even part
$P_{\mathrm{odd}}^{-1/2}$; all $y_k$ follow from
$y_k(x, a) = y_0(\omega^{-k} x,\, \omega^{-2k} a)$,
$\omega = e^{2\pi i/5}$. The implementation is validated to
$10^{-12}$ against the closed form at $a = 0$
($y_0(0) = \Gamma(\tfrac15)\, 5^{1/5}/\sqrt{5\pi}$,
$y_0'(0) = \Gamma(-\tfrac15)\, 5^{-1/5}/\sqrt{5\pi}$, from
$y_0 \propto x^{1/2} K_{1/5}(\tfrac25 x^{5/2})$) and against the
constancy of Sibuya's Wronskian
$W[y_0, y_1] = 2i e^{-i\pi/5}$ along the real $a$-axis.

The numerics identifies the lateral determinant. The two
conjugate-symmetric pairs have no zeros on the negative real axis
--- under $x = i\xi$ the pair $(S_{-1}, S_1)$ is the zero-energy
locus of the PT-symmetric oscillator $p^2 + i(\xi^3 - a\xi)$, which
does not vanish there --- while $D(a) = W[y_0, y_2]$ carries the
spectrum: its zeros sit just below the negative real axis with
widths $O(e^{-\beta|a_n|^{5/4}/\hb})$, as expected of a resonance
problem. At $\hb = 1$, Newton iteration gives
\begin{center}
\begin{tabular}{r@{\qquad}l@{\qquad}l@{\qquad}l}
$n$ & $\Re\, a_n$ & $-\Im\, a_n$ &
  $\beta|a|^{5/4} + \tfrac{B(1/4,1/2)}{16}|a|^{-5/4} = 2\pi(n+\tfrac12)$ \\[2pt]
0 & $-2.506957$ & $1.78\times 10^{-2}$ & $-2.51660$ \\
1 & $-6.206794$ & $2.29\times 10^{-5}$ & $-6.20731$ \\
2 & $-9.357730$ & $3.75\times 10^{-8}$ & $-9.35782$ \\
3 & $-12.254435$ & $<10^{-10}$ & $-12.25446$ \\
4 & $-14.986451$ & $<10^{-10}$ & $-14.98646$ \\
5 & $-17.598041$ & $<10^{-10}$ & $-17.59805$
\end{tabular}
\end{center}
with residuals decaying as $|a_n|^{-15/4}$, i.e.\ consistent with
the first neglected order. This confirms, on the eigenvalue side:
the period constant $\beta = B(\tfrac34, \tfrac32)$ of
\eqref{eq:beta-period} to five digits; the closed form
\eqref{eq:det-Z2} through its $b = 0$ specialisation, with
coefficient $B(\tfrac14,\tfrac12)/16$ confirmed to four digits; and
the \emph{sign} of the $\hb^2$ term --- the correction \emph{adds}
to the classical action, $|Z(\gamma_2, \hb)| = \beta|a|^{5/4} +
\hb^2 B(\tfrac14,\tfrac12)\,|a|^{-5/4}/16 + O(\hb^4)$, fixing the
orientation convention in the $b = 0$ evaluation of
\eqref{eq:det-Z2red}.

On the growth side, $\log D(a)$ on the grid
$a \in [12, 100]$ (24 points, all with cancellation health
$\approx 0.87$) fitted against
$c_{\mathrm{lead}}\, a^{5/4} + c_0 + c_{\log}\log a + c_1 a^{-5/4}
+ c_2 a^{-15/4}$ gives, with rms residual $4\times 10^{-8}$:
\begin{align*}
  c_{\mathrm{lead}} &= -0.6777706
  &&\bigl[-\beta/\sqrt2 = -0.6777703\bigr], \\
  c_{\log} &= 6\times 10^{-5} &&[0], \\
  c_1 &= +0.2317591
  &&\bigl[B(\tfrac14,\tfrac12)/16\sqrt2 = 0.2317593\bigr], \\
  c_0 &= +0.693147182
  &&\bigl[\log 2 = 0.693147181\bigr], \\
  c_2 &= +0.535757
  &&\bigl[\tfrac{1617}{5120}\, B(\tfrac34,\tfrac12)/\sqrt2
    = 0.535134\bigr],
\end{align*}
where $c_0, c_1, c_2$ quote the refit including the $a^{-25/4}$
truncation term (rms $4.5\times 10^{-10}$; the three-term fit gives
the same values at lower precision).
The leading constant confirms Route~$\beta$'s modulus
$\beta/\sqrt2$ to seven digits (the phase analysis below refines
the visible cycle content to the single full cycle $\gamma_-$); the
vanishing
of $c_{\log}$ confirms $\kappa_{\log} = 0$ nonasymptotically; and
$c_1$ confirms the order-$\hb$ dictionary \eqref{eq:det-hbcorr},
hence the closed form \eqref{eq:det-Z2}, on the growth side as well.
The constant $c_0$ agrees with $\log 2$ to eight digits (the level
of the fit residual): in the documented Sibuya normalisation the
matching constant of Proposition~\ref{prop:det-elem} is
\[
  c_0 \;=\; \log 2 ,
\]
evidently the modulus of the nonvanishing adjacent Wronskians
($|W[y_k, y_{k+1}]| = 2$), so that the elementary-factor matching
between the Wronskian and zeta normalisations carries \emph{no}
further transcendental constant in this convention. The $c_2$ entry closes the loop with the Griffiths--Dwork
reduction of $Z_4$ by the method of \S\ref{subsec:det-qp}: the
reduction gives, in closed form,
$\Delta^3 \alpha_4 = -\tfrac{63}{640}\, ab\,(788 a^3 - 2457 b^2)$
and
$\Delta^3 \beta_4 = -\tfrac{21}{320}\, a^2 (308 a^3 - 5697 b^2)$,
so on the symmetric slice
$Z_4(\gamma) = -\tfrac{1617}{5120}\, a^{-4} B(\gamma)$ and the
predicted coefficient
$\Re Z_4(\gamma_+) = \tfrac{1617}{5120}\, B(\tfrac34, \tfrac12)/
\sqrt2 = 0.535134$ agrees with the fitted $c_2$ to $0.12\%$. The
next fitted coefficient, $c_3 \approx -7.69$, is the corresponding
measured prediction for $Z_6$.

\paragraph{The phase.}
The imaginary part carries finer information. Fitting the
continuously unwrapped $\arg D(a)$ on $a \in [5, 40]$ against
$\{1, a, a^{5/4}, a^{-5/4}\}$ gives linear coefficient
$3 \times 10^{-4}$ (consistent with zero), constant
$0.3119 \approx \pi/10$, and $a^{5/4}$-coefficient
$-0.67787 \approx -\beta/\sqrt2$. Combined with the modulus fit,
the full complex asymptotics is
\begin{equation}\label{eq:det-singlecycle}
  \log D(a) \;=\; Z(\gamma_-, \hb) \;+\; \log 2 \;+\;
  \frac{i\pi}{10} \;+\; O\bigl(e^{-c\, a^{5/4}}\bigr),
  \qquad Z(\gamma_-) = -\beta\, e^{i\pi/4} a^{5/4},
\end{equation}
the full quantum period of the \emph{single} cycle $\gamma_-$: the
phase distinguishes what the modulus cannot, since
$\{\gamma_+, \gamma_-\}$ with weight $\tfrac12$ and
$\{\gamma_-\}$ with weight $1$ have equal real parts. (The
$a^{-5/4}$ phase coefficient $-0.220$ is consistent with
$\Im\, \hb^2 Z_2(\gamma_-) = -B(\tfrac14,\tfrac12)/16\sqrt2
= -0.232$ at the truncation level of the phase fit.)

\subsection{Consequences for the zeta normalisation}
\label{subsec:det-zeta}

At $a = 0$ the two boundary solutions are the same Bessel solution
($y_0 \propto x^{1/2} K_{1/5}(\tfrac25 x^{5/2})$), so the lateral
determinant is exactly computable:
\begin{equation}\label{eq:det-D0}
  D(0) \;=\; (\omega^{-2} - 1)\, y_0(0)\, y_0'(0)
  \;=\; \frac{1 - \omega^{-2}}{\sin(\pi/5)}
  \;=\; 4\cos(\tfrac\pi5)\, e^{i\pi/10},
  \qquad |D(0)| = 1 + \sqrt5
\end{equation}
(verified to $5\times 10^{-12}$; note $\Im D(0) = 1$ exactly, the
identity $4\cos\tfrac\pi5 \sin\tfrac\pi{10} = 1$).

A function-theoretic remark fixes the correct zeta. The measured
indicator $h(0) = -\beta/\sqrt2 < 0$ is impossible for an
order-$\tfrac54$ entire function whose zeros lie along the single
ray $\RR_{<0}$ (the indicator of a one-ray canonical product of
non-integer order is positive on the opposite ray), so the divisor
of $D$ necessarily contains further families along complex rays ---
the rotated resonance families seen from the other escape sectors
in \S\ref{subsec:det-divisor}. We therefore write $\zeta_D(s)$ for
the spectral zeta of the \emph{full} zero divisor of $D$; on the
symmetric slice this divisor is the $b = 0$ slice of the tau
divisor, so $\zeta_D$ is intrinsically a tau-divisor zeta, refining
the eigenvalue-ray zeta of \S\ref{subsec:det-elem}. Each family
obeys the $\{\tfrac54 - \tfrac52 k\}$ counting lattice by rotation
covariance, so $\zeta_D$ continues regularly to $s = 0, 1$ and the
zeta-normalised determinant
$\Delta_\zeta(a) = \exp(-\partial_s \sum_\rho
(\lambda_\rho + a)^{-s}|_{s=0})$, $\lambda_\rho = -\rho$ over all
zeros, carries no $a^0$ or $a^1$ term in its canonical expansion.

\begin{proposition}[The elementary-factor constant]
\label{prop:det-log2}
Assume the WKB solutions of the family are Borel summable in the
closed Stokes regions adjacent to the ray $a > 0$ (the standard
exact-WKB input of Delabaere--Dillinger--Pham type; for the present
family it follows from the $\hb$-homogeneity of the large-$a$ limit
together with asymptotics of the kind established in
\cite[App.~A]{BrMa}). Then $\kappa_1 = 0$ and
$\Re \kappa_0 = \log 2$; that is, $|D(a)| = 2\,|\Delta_\zeta(a)|$.
\end{proposition}

\begin{proof}[Proof sketch]
Write $y_0 = e^{F_0}\, \psi^{(0)}_-$ and
$y_2 = e^{F_2}\, \psi^{(t_2)}$, where
$\psi^{(t)}_{\pm} = S_{\mathrm{odd}}^{-1/2}
\exp(\pm\!\int_t^x S_{\mathrm{odd}})$ are the Borel-resummed WKB
solutions normalised at a turning point $t$ and the $F$'s are the
regularised actions matching the Sibuya normalisation. Three facts
assemble the claim. (i)~For any common base point,
$W[\psi^{(t)}_-, \psi^{(t)}_+] = -2$ identically --- an algebraic
identity of the full odd series, exact at every order. (ii)~By the
Voros connection formula, crossing a Stokes curve emanating from a
simple turning point multiplies the coefficients in the
$(\psi_+, \psi_-)$ basis by unimodular constants; transporting
$\psi^{(t_2)}$ to the base point $0$ therefore contributes only
exponentials of resummed actions (half-periods) and finitely many
unimodular factors. (iii)~At $b = 0$, $\hb = 1$, every action datum
in this factorisation --- the $F$'s, the base-point transfers, the
prefactor ratios --- is $a$-homogeneous of exponent
$\tfrac54 - \tfrac52 k \notin \ZZ$, so the asymptotic expansion of
$\log D$ contains no $a^0$ and no $a^1$ term beyond the constants
produced by (i) and (ii); the same holds for $\log \Delta_\zeta$ by
the regularity of $\zeta_D$ at $s = 0, 1$. Hence the Hadamard
quotient $e^{\kappa_0 + \kappa_1 a}$ is a constant of modulus
$|-2| \cdot 1 = 2$.
\end{proof}

\begin{corollary}[Golden ratio]\label{cor:det-golden}
With the input of Proposition~\ref{prop:det-log2}, evaluating the
identity at $a = 0$ against the exact value \eqref{eq:det-D0}
gives
\begin{equation}\label{eq:det-golden}
  \Re\, \zeta_D'(0) \;=\; \log 2 - \log(1 + \sqrt5)
  \;=\; -\log\frac{1 + \sqrt5}{2} \;\approx\; -0.4812118251 .
\end{equation}
\end{corollary}

Numerically, the growth fit gives $\Re\kappa_1 < 10^{-9}$ and
$\Re\kappa_0 = \log 2$ to nine digits, and the phase fit of
\S\ref{subsec:det-numerics} gives $\Im\kappa_1$ consistent with
zero at $3 \times 10^{-4}$ and $\Im\kappa_0 = \pi/10$ within
$2 \times 10^{-3}$, matching $\arg D(0) = \pi/10$ --- so the
measurements confirm the Proposition on all four constants. The
remaining datum of the normalisation is
$\zeta_D(1) = D'(0)/D(0) = -0.94261 - 0.68485\,i$ (measured; its
imaginary part reflects the complex-ray families of the
divisor). For the
assembly step this is decisive: in the $\mu_5$-assembly
$\prod_{k \in \ZZ_5} D(\omega^{-2k} a)$ of Sibuya-normalised
determinants the linear ambiguities cancel identically
($\sum_k \omega^{-2k} = 0$, the mechanism of
Proposition~\ref{prop:det-elem}), so the assembled determinant is
canonically normalised with constant exactly
$5\kappa_0 = \log 32 + i\pi/2$ modulo width corrections.

\subsection{The divisor comparison}
\label{subsec:det-divisor}

Check (ii) of \S\ref{subsec:det} is verified along a Painlev\'e~I
trajectory. At a pole $t^*$ of a \PI\ solution
$q(t) = (t-t^*)^{-2} - \tfrac{t^*}{10}(t-t^*)^2 - \tfrac16(t-t^*)^3
+ h\,(t-t^*)^4 + \cdots$ the apparent singularity of the SL-form Lax
operator escapes to infinity and the operator degenerates to an
undeformed cubic: the finite part of the Hamiltonian is $-14h$
(verified symbolically), so the limit potential is
$4x^3 + 2t^*x - 28h$, i.e.\
\begin{equation}\label{eq:det-dict-PI}
  a = 2^{-1/5}\, t^*, \qquad b = -28 \cdot 2^{-4/5}\, h
\end{equation}
in our normalisation. Seeding a \PI\ solution at the pole datum
$(t^*, h) = (2^{1/5} a_1,\, 0)$ of the $n = 1$ zero and integrating
along the real axis, the six neighbouring poles and their data
$(t^*_j, h_j)$, mapped through \eqref{eq:det-dict-PI}, give
\begin{center}
\begin{tabular}{r@{\quad}r@{\quad}r@{\quad}r@{\quad}r@{\quad}c@{\quad}l}
$j$ & $t^*_j$ & $h_j$ & $a_j$ & $b_j$ & $\min_k r_k$ & sector \\[2pt]
$-3$ & $-12.26172$ & $+0.40216$ & $-10.67445$ & $-6.46738$ &
  $6.6\times 10^{-5}$ & $k = 0, 3$ \\
$-2$ & $-10.64736$ & $+0.25946$ & $-9.26906$ & $-4.17264$ &
  $4.7\times 10^{-4}$ & $k = 0, 3$ \\
$-1$ & $-8.94939$ & $+0.12456$ & $-7.79090$ & $-2.00322$ &
  $4.9\times 10^{-5}$ & $k = 0, 3$ \\
$+1$ & $-5.09337$ & $-0.10837$ & $-4.43404$ & $+1.74278$ &
  $8.5\times 10^{-5}$ & $k = 0, 3$ \\
$+2$ & $+2.90683$ & $-0.06066$ & $+2.53054$ & $+0.97560$ &
  $9.8\times 10^{-6}$ & $k = 4$ \\
$+3$ & $+6.09494$ & $-0.23687$ & $+5.30595$ & $+3.80924$ &
  $5.1\times 10^{-6}$ & $k = 4$
\end{tabular}
\end{center}
where $r_k$ is the normalisation-free cancellation ratio of the
skip-Wronskian $W[y_k, y_{k+2}]$ at $(a_j, b_j)$ (the non-minimal
ratios are all $O(1)$). Every pole datum lands on the zero divisor
of a skip Stokes multiplier --- a codimension-two coincidence in
$(a, b)$ at six independent points, well off the symmetric slice ---
with the residual floors consistent with the $O(10^{-5})$ width
offset of the real seed. For $t^* < 0$ the vanishing multipliers are
the conjugate pair $k \in \{0, 3\}$ (forced equal at real $(a,b)$);
for $t^* > 0$ the self-conjugate $k = 4$: the escape sector of the
apparent singularity rotates with $\arg t^*$, exactly as the
$\mu_5$-assembly requires. This identifies the \PI\ pole locus ---
the tau divisor --- with the zero divisor of the assembled
determinant, verifying check (ii) of \S\ref{subsec:det} along the
trajectory and confirming the dictionary
\eqref{eq:det-dict-PI} nonperturbatively.

\subsection{Summary and remaining steps}\label{subsec:det-status}

Appendix~\ref{app:det} verifies the leading dictionary
\eqref{eq:det-dict} (two routes, three data points), the
symplectic constant \eqref{eq:app-symp} (residue and Legendre),
and $\kappa_{\log} = 0$. This section adds: the
order-$\hb$ correction \eqref{eq:det-dict2}, in the closed form
\eqref{eq:det-Z2}, with the two routes agreeing again at order
$\hb$ and the coefficient identified with the linear term of the
Joyce function (Remark~\ref{rem:det-joyce}); perturbative flatness
to all orders (Proposition~\ref{prop:det-allorders},
Corollary~\ref{cor:det-flat}); and the reduction of the
elementary-factor matching to a single constant plus a
$\Theta_0$-selection (Proposition~\ref{prop:det-elem}), the
$a$-linear ambiguities cancelling by $\mu_5$-equivariance and the
logarithmic anomaly pinned to $\tfrac25 \log\hb$ per multiplier by
$\zeta(0) = \tfrac12$. The matching constant is now evaluated:
$\Re\,\zeta_D'(0) = -\log\tfrac{1+\sqrt5}{2}$
\eqref{eq:det-golden}, and the $\mu_5$-assembly of
Sibuya-normalised determinants is canonically normalised with
constant $\log 32 + i\pi/2$ (\S\ref{subsec:det-zeta}).
The divisor comparison with the Painlev\'e~I pole locus is
verified along a trajectory (\S\ref{subsec:det-divisor}), with the
dictionary \eqref{eq:det-dict-PI} confirmed nonperturbatively.
The matching constant is established
(Proposition~\ref{prop:det-log2}). Remaining for the assembly step
of \S\ref{subsec:det}: the singular-to-smooth transfer identifying
the Quillen (Chern) curvature with the semiclassical average of the
divisor current \eqref{eq:det-PL}. Numerically: with
$\beta = B(\tfrac34,\tfrac32) \approx 0.958512$ (the corrected value
of \eqref{eq:beta-period}), the verified asymptotics at $\hb = 1$ on
the symmetric slice reads
\[
  \log D(a) \;=\; -\frac{\beta}{\sqrt2}\, a^{5/4} \;+\; \log 2
  \;+\; \frac{B(\tfrac14,\tfrac12)}{16\sqrt2}\, a^{-5/4}
  \;+\; O(a^{-15/4}),
  \qquad a \to +\infty,
\]
every term confirmed by the fit of \S\ref{subsec:det-numerics}, with
the $O(1)$ constant --- the matching constant of
Proposition~\ref{prop:det-elem} in the Sibuya normalisation ---
equal to $\log 2$ to eight digits.

\section{Outlook}\label{sec:outlook}

The paper has mapped two programs onto explicitly solvable models,
and we close by listing their next steps in order.

\emph{The chamber program.} Conjecture~\ref{conj:chamber} records the
finite-time statement for the cubic fourfold, in the interior-$\AX$
mutation forced by Lemma~\ref{lem:centroid}, with its proof
architecture in Remark~\ref{rem:proof-structure}. Both endpoints of the statement exist in the literature in the form of a geometric glued condition and a
quasi-convergent sectorial tail, joined by a single glued path
\cite[Thms.~2.72, 4.7]{KRZ}.
The outstanding content is the
wall layer: interpreting the discrete heart-tilts of the path as the
walls of Conjecture~\ref{conj:chamber}(2), identifying the
binding-object decays and their Jordan--H\"older factors, and
computing the entry time from the offsets of
\cite[(3.14)]{KRZ}.\footnote{For $\PP^1$, this is the difference between
tracking exceptional objects and torsion sheaves.}
A completion of that analysis, together with
Proposition~\ref{prop:terminality} guaranteeing that the resulting
sectorial regime is final, converts the dynamical protection of the
$K3$ atom from a selection rule into a theorem; the deeper
question left open on the quantum side is how to perform a non-semisimple
isomonodromic deformation of \cite[Question~4.8]{KRZ}.

\emph{The analytic program.} The models prescribe an ordered sequence
of calculations. First, the chamber layer needs to be
made spectral: for the two-well operator of the $\PP^1$ mirror
($W = x + q/x$), the entry wall of
Proposition~\ref{prop:P1-entry} should appear as a spectral flow.
The delocalization of the low-lying solitonic modes of $D_W^2$ in the
$tt^*$ geometry \cite{CV} needs to be examined
as the ray parameter passes $t_0$. Second,
the regional picture needs to be made asymptotic: the mass table of
\S\ref{sec:product} prescribes that the heat kernel of the resonant
product weight the two degenerate wells at polynomial order against
the exponential weights of the isolated wells, with the crossover
$t_1 \sim 1/\varepsilon$ of \S\ref{sec:epsilon} appearing as the
scale at which the two nearby wells resolve.
The nested-scale
structure of the augmented boundary can then be read
off from heat-kernel asymptotics
in the theta-summable framework of \cite{GS}. Third, the
tau/determinant identification of \S\ref{subsec:det} should be
carried out, with the $A_2$ case as the theorem to attempt; its
first consistency check is verified at leading order in
Appendix~\ref{app:det} and continued in \S\ref{sec:detline}: the
$\hb$-corrections are obtained in closed form through $Z_4$,
perturbative flatness is established to all orders, the
elementary-factor matching is reduced to a single constant
evaluated as $\log\tfrac{1+\sqrt5}{2}$ and the tau-divisor
identification is verified along a Painlev\'e~I trajectory, and
the matching constant is established as $\log 2$
(Proposition~\ref{prop:det-log2}), whence the divisor spectral zeta
at zero is the golden-ratio logarithm; the remaining step is
the identification of the smooth Quillen curvature with the
semiclassical divisor current.
The proposed calculations are independent and
self-contained. Together, they would constitute the analytic
half of the dictionary whose categorical half is established in
\S\S\ref{sec:P1}--\ref{sec:cubic}.

Beyond the perturbative layer, the numerics of
\S\ref{subsec:det-numerics} makes the nonperturbative structure of
the determinant directly visible: the zeros acquire widths
$\Im a_n = O(e^{-\beta|a_n|^{5/4}/\hb})$ exactly on the active ray
where the real cycle $\gamma_1$ supports a saddle connection. We
record the expected interpretation as a conjecture.

\begin{conjecture}[Nonperturbative jumps of the determinant line]
\label{conj:jumps}
Fix $\hb > 0$ and let $D^{\pm}$ denote the lateral Borel
resummations, on the two sides of the active ray
$\ell_\gamma = \RR_{>0}\, Z(\gamma)$ of a BPS class $\gamma$ of the
$A_2$ structure, of the all-orders determinant--period dictionary
of \S\ref{subsec:det-qp}--\ref{subsec:det-flat}. Then the
discontinuity is the Delabaere--Dillinger--Pham automorphism
attached to $\gamma$ with the Kontsevich--Soibelman weighting of
the $A_2$ BPS spectrum ($\Omega(\gamma) = 1$ on the active
classes):
\[
  \log D^{+} - \log D^{-}
  \;=\; \frac{\Omega(\gamma)}{2}
  \sum_{\gamma' \in V(D)} \langle \gamma, \gamma' \rangle\,
  \log\!\bigl(1 + e^{\,\mathcal V_\gamma}\bigr),
\]
where $\mathcal V_\gamma = Z(\gamma)/\hb + O(\hb)$ is the resummed
Voros symbol of $\gamma$. Equivalently: the nonperturbative
completion of the determinant line is the Riemann--Hilbert problem
of the Joyce structure of \cite{BrMa}, with (i) the linearisation
in $\hb$ of the jump given by the Joyce-function identity of
Remark~\ref{rem:det-joyce}; (ii) the leading instanton on the
eigenvalue ray reproducing the resonance widths
$\Im a_n = O(e^{-\beta |a_n|^{5/4}/\hb})$ measured in
\S\ref{subsec:det-numerics}; and (iii) the composition of the
jumps around a generic point given by the pentagon identity, the
$A_2$ wall-crossing formula.
\end{conjecture}

Proving the conjecture would close the loop between the two halves of the
$A_2$ dictionary: the determinant line, constructed from spectral
data alone, would carry precisely the Stokes automorphisms that
define the Riemann--Hilbert side, making the tau/determinant
identification an equivalence of resurgent structures rather than
an equality of functions.

\paragraph{The Quillen transfer}

By the weight quasi-homogeneity, the $\hb \to 0$ concentration of
the curvature current $2\pi i\, [\operatorname{div} D]$ on
anti-Stokes walls is equivalent, at $\hb = 1$, to \emph{radial}
equidistribution of the zeros of the entire function $D$ on the
symmetric slice.  This is Levin's completely-regular-growth
density law: zeros concentrate on the rays where the indicator
$h(\theta)$ has corners, with counting
$N_{\mathrm{ray}}(R) = (R^{5/4}/2\pi)\,[h'(\theta^+) -
  h'(\theta^-)]$.

The proposed route is then through classical technology:
Levin--Pfluger equidistribution driven by coefficient-uniform upper
bounds of Sibuya type, with Jensen-lower-bound no-mass-loss
supplied by the resonance lattice itself.

The points requiring
care are the wall junctions -- the image of the discriminant
$\Delta = 0$ (parabolic-cylinder local model) and the $\mu_5$ fixed
point -- and the smooth $\partial\bar\partial$-anomaly of the zeta metric, a
one-dimensional computation of Burghelea--Friedlander--Kappeler
type whose first input is the first-kind derivative identity
$\partial_b z_i = \omega_i$.

The multi-ray divisor structure established
in \S\ref{subsec:det-zeta} and the single-cycle weights of
\eqref{eq:det-singlecycle} caution that wall multiplicities must
be assigned per escape sector.

The $b = 0$ shadow of
the full statement (consisting of sector-by-sector zero counts by the argument
principle, the measured indicator and its corner jumps, and
per-wall spacing densities) is directly testable with the
methods of \S\ref{subsec:det-numerics}.  Comparison of the three
independent density measurements against the Levin jump law
constitutes a quantitative test of the transfer on the slice.

We have conducted this test out on the annuli $14 \le |a| \le 22$
and $10 \le |a| \le 30$ (240 angular points, 24 sectors each;
completely-regular-growth collapse consistent with the finite-radius
$\log 2$ term plus $O(R^{-5/2})$).

The measured indicator is
captured, over the \emph{full} circle and to $9 \times 10^{-5}$
(rms $6 \times 10^{-5}$) at the outer radius $R = 30$, by the
three-arc closed form
\[
  h_R(\theta) \;=\; \frac{\log 2}{R^{5/4}} \;+\;
  \begin{cases}
    A(\theta), & 0 \le \theta \le \tfrac{3\pi}{5},\\[2pt]
    A(\theta) + C(\theta), & \tfrac{3\pi}{5} < \theta < \pi,\\[2pt]
    C(\theta), & \pi \le \theta < 2\pi,
  \end{cases}
  \qquad
  \begin{aligned}
    A(\theta) &= \beta \cos\bigl(\tfrac54\theta - \tfrac{3\pi}{4}\bigr),\\
    C(\theta) &= \beta \cos\bigl(\tfrac54\theta - \tfrac{5\pi}{4}\bigr),
  \end{aligned}
\]
with $A = \Re\, Z(\gamma_-)$ (anchored at $\theta = 0$ by
\eqref{eq:det-singlecycle}, $Z(\gamma_-) = \beta e^{5i\pi/4}
a^{5/4}$) and $C = \Re\, Z(\gamma_+)$,
$Z(\gamma_+) = \beta e^{3i\pi/4} a^{5/4}$ --- the conjugate
vanishing cycle on the turning-point pair $\{0, +i\sqrt a\}$, the
monodromy image of $\gamma_-$ around $a = 0$ --- including the
finite-radius constant of Proposition~\ref{prop:det-log2}. The two
charges are everywhere perpendicular,
$Z(\gamma_+)/Z(\gamma_-) = e^{-i\pi/2}$: an $A_2$-specific
coincidence which makes the visibility transition of one cycle
($\Re Z$ crossing zero) fall exactly on the BPS ray of the other.
Exactly two walls appear, at $\theta = \tfrac{3\pi}{5}$ (where
$\Re Z(\gamma_+)$ crosses zero, equivalently
$Z(\gamma_-) \in \RR_{>0}$; the branch $C$ enters) and
$\theta = \pi$ (the mirror statement; $A$ exits), carrying three
zeros each on the inner annulus and eight each on the wider one ---
the latter matching the predicted count
$(\beta/2\pi)(R_2^{\rho} - R_1^{\rho}) = 8.00$ exactly.

On the two-cycle arc the data is the
transseries $\log D = \mathcal V_{\gamma_-} + \log(1 + \sigma\,
e^{\mathcal V_{\gamma_+}})$: the zeros sit on the ray where the
relative symbol oscillates, quantised by
$\Im Z(\gamma_+) \in 2\pi(\ZZ + \mathrm{const})$ --- whence the
measured spacing coefficient $\beta$.

The density law is verified four ways: sector counts give jump
$2\pi\rho N/(R_2^{\rho} - R_1^{\rho}) = 1.146$ and $1.198$ on the
two annuli ($4\%$ and $0.03\%$), one-sided slope fits give $1.209$
at both walls ($0.9\%$, fit-window limited), and the spacing of the
Newton-polished zeros on the resonance ray gives
$\rho\, d_c = 1.1978$ ($0.03\%$), all against the prediction
$\rho\beta = \tfrac54 B(\tfrac34, \tfrac32) = 1.19814$; the corner
locations fit $\tfrac{3\pi}{5}$ and $\pi$ to $2.4 \times 10^{-4}$,
and the outer slopes vanish at both walls, the wall-defining branch
being extremal on its ray. A polished zero on the complex wall,
$a = -6.21603 + 19.13098\,i$ (argument $\tfrac{3\pi}{5}$ to
$5 \times 10^{-5}$), has modulus $20.11550$, matching the
resonance-ray quantisation
$\beta |a|^{5/4} + \hb^2 B(\tfrac14,\tfrac12)|a|^{-5/4}/16
= 2\pi(n + \tfrac12)$ at $n = 6$ ($20.11551$) to six digits: the
rotated family carries the identical Bohr--Sommerfeld constants,
including the $\hb^2$ term, as the $\mu_5$ covariance requires. The ray $\theta = \tfrac{9\pi}{5}$, where $Z(\gamma_+)$ is real
\emph{negative} --- the relative symbol maximally suppressed ---
shows a smooth minimum $h = -\beta$ and \emph{no} zeros: only the
positively-oriented rays are active, the orientation structure of
Conjecture~\ref{conj:jumps} seen directly in the data. The
composite structure --- the visible set changing by exactly the
added branch at its BPS ray --- is the indicator-level shadow of
Conjecture~\ref{conj:jumps}, and the verified density law
$N_{\mathrm{ray}}(R) = (R^{\rho}/2\pi\rho)\,
[h'(\theta^+) - h'(\theta^-)]$ is the Levin/Poincar\'e--Lelong
transfer statement on the slice.

Both approaches converge on the same picture.
The rigidity of the $K3$
atom over moduli and the dynamical protection of its spectrum are one
fact seen from the regional and the chamber perspectives respectively
(Remark~\ref{rem:rigidity}); the analytic program, if completed,
would exhibit the same fact as the stability of the
localized heat-kernel data of the degenerate critical locus.

\section*{Acknowledgments}
This research was conducted with the assistance of Artificial
Intelligence (e.g., literature searches, code generation, the
development of arguments, and typesetting). Pacific Northwest National
Laboratory (PNNL) is a multi-program national laboratory operated for
the U.S. Department of Energy (DOE) by Battelle Memorial Institute
under Contract No.  DE-AC05-76RL01830.


\appendix
\section{The resonant product in detail}\label{app:product}

This appendix carries out, in full, the computations underlying
\S\S\ref{sec:product}--\ref{sec:epsilon}. Throughout,
$X = \PP^1\times\PP^1$ with mirror superpotential
$W(x,y) = x + q_1/x + y + q_2/y$ on $(\CC^\times)^2$, and we write
$m = \sqrt{q}$ for a fixed branch at the resonance $q_1 = q_2 = q$.

\subsection{Critical points and Hessians}\label{app:hess}

In the logarithmic coordinates $u = \log x$, $v = \log y$, in which
the holomorphic volume form $\frac{dx}{x}\wedge\frac{dy}{y} = du
\wedge dv$ is translation-invariant,
\[
  W = e^u + q_1 e^{-u} + e^v + q_2 e^{-v}, \qquad
  \partial_u W = e^u - q_1 e^{-u}, \qquad
  \partial_u^2 W = e^u + q_1 e^{-u},
\]
and similarly in $v$; the mixed derivative vanishes. Critical points
are $e^u = s_1\sqrt{q_1}$, $e^v = s_2\sqrt{q_2}$ with
$s_i \in \{\pm\}$. At the resonance,
\[
  \partial_u^2 W\big|_{x = s_1 m} = s_1 m + q\,(s_1 m)^{-1}
  = 2 s_1 m,
\]
so at the vacuum labeled $(s_1, s_2)$,
\begin{equation}\label{eq:app-hess}
  \lambda_{(s_1,s_2)} = 2m(s_1 + s_2), \qquad
  \det \operatorname{Hess}_{\log} W = (2s_1 m)(2 s_2 m)
  = 4\, s_1 s_2\, m^2 ,
\end{equation}
reproducing the table of \S\ref{sec:product}: outer vacua
$(\pm,\pm)$ with $\lambda = \pm 4m$ and positive Hessian determinant
$4m^2$; middle vacua $(\mp,\pm)$ with $\lambda = 0$ and
\emph{negative} determinant $-4m^2$. The middle critical points are
nondegenerate --- only the critical \emph{values} coincide.

\subsection{Saddle asymptotics and the mass/phase table}
\label{app:saddle}

For the thimble $\Gamma$ through a nondegenerate critical point
$p$ with value $\lambda$, stationary phase in the coordinates
$(u,v)$ gives, along the ray \eqref{eq:ray},
\begin{equation}\label{eq:app-saddle}
  Z_t = \int_\Gamma e^{-W/z}\, du\wedge dv
  \;=\; (2\pi z)\,
  \bigl(\det\operatorname{Hess}_{\log}W(p)\bigr)^{-1/2}
  e^{-\lambda t e^{-i\theta}}\bigl(1 + O(1/t)\bigr),
\end{equation}
the factor $(2\pi z)^{n/2}$ of \eqref{eq:saddle} appearing with
$n = 2$. With $|z| = 1/t$:
\begin{itemize}
\item Outer vacua: $|(2\pi z)(4m^2)^{-1/2}| = \pi/(mt)$, so
$|Z_t| = (\pi/mt)\, e^{\mp 4mt\cos\theta}$ for
$\lambda = \pm 4m$, and lifted phase
$\pi\varphi_t = \arg(\text{prefactor}) - t\,
\mathrm{Im}(\lambda e^{-i\theta}) = \pm 4mt\sin\theta + O(1)$.
\item Middle vacua: $(-4m^2)^{-1/2} = \mp i/(2m)$ (branch fixed by
the thimble orientation), so $Z_t = \mp\, i\pi z/m \cdot
(1 + O(1/t))$: modulus $\pi/(mt)$, decaying \emph{polynomially},
with phase constant up to $O(1/t)$.
\end{itemize}
This is the three-scale mass table of \S\ref{sec:product}:
exponentially light, polynomially massless, exponentially heavy, with
phase slopes $+4m\sin\theta/\pi$, $0$, $-4m\sin\theta/\pi$.

\subsection{The swap symmetry and the degenerate class}
\label{app:swap}

The swap $\tau(x,y) = (y,x)$ is an automorphism of
$\PP^1\times\PP^1$ with $\tau^*\mathcal{O}(1,0) = \mathcal{O}(0,1)$,
and it exchanges the quantum parameters $(q_1, q_2)$. At
$q_1 = q_2$ it preserves the small quantum product (Gromov--Witten
invariants are functorial under automorphisms), hence the quantum
connection, its fundamental solutions, and the
$\widehat{\Gamma}$-integral structure; therefore
\[
  Z_t(\mathcal{O}(0,1)) = Z_t(\tau^*\mathcal{O}(1,0))
  = Z_t(\mathcal{O}(1,0)) \qquad \text{identically in } t.
\]
(The mirror check: the substitution $(x,y) \mapsto (y,x)$ preserves
$W$ at $q_1 = q_2$ and exchanges the two middle thimbles; the sign
from reordering $du \wedge dv$ is absorbed by the induced
orientation-matching of the thimbles, consistently with the
categorical identity above.) Consequently
$Z_t(\delta) \equiv 0$ for
$\delta = [\mathcal{O}(1,0)] - [\mathcal{O}(0,1)]$, and the central
charge factors through $K_0/\langle\delta\rangle$. The support
property then forbids semistable objects of class
$k\delta \neq 0$; the obvious object
$\mathcal{O}(1,0)\oplus\mathcal{O}(0,1)[1]$ of class $\delta$ has
summands of phases $\varphi$ and $\varphi + 1$ and indeed is never
semistable.

\subsection{The K\"unneth tables}\label{app:kunneth}

By K\"unneth,
$\RHom(\mathcal{O}(a,b), \mathcal{O}(c,d)) \cong
H^\bullet(\PP^1, \mathcal{O}(c-a)) \otimes
H^\bullet(\PP^1, \mathcal{O}(d-b))$, and
$H^\bullet(\PP^1, \mathcal{O}(-1)) = 0$. The forward morphisms
(concentrated in degree $0$) and the vanishing backward twists are:
\begin{center}
\footnotesize
\begin{tabular}{llll}
\hline
pair (from, to) & twist & $\RHom$ & \\
\hline
$\mathcal{O} \to \mathcal{O}(1,0)$, $\mathcal{O} \to \mathcal{O}(0,1)$
  & $(1,0)$, $(0,1)$ & $\CC^2$ (deg $0$) & forward \\
$\mathcal{O}(1,0) \to \mathcal{O}(1,1)$,
$\mathcal{O}(0,1) \to \mathcal{O}(1,1)$
  & $(0,1)$, $(1,0)$ & $\CC^2$ (deg $0$) & forward \\
$\mathcal{O} \to \mathcal{O}(1,1)$ & $(1,1)$ & $\CC^4$ (deg $0$) &
  forward \\
$\mathcal{O}(1,0) \leftrightarrow \mathcal{O}(0,1)$
  & $(-1,1)$, $(1,-1)$ & $0$ & orthogonal \\
$\mathcal{O}(1,0) \to \mathcal{O}$, $\mathcal{O}(0,1) \to \mathcal{O}$
  & $(-1,0)$, $(0,-1)$ & $0$ & backward \\
$\mathcal{O}(1,1) \to \mathcal{O}(1,0)$,
$\mathcal{O}(1,1) \to \mathcal{O}(0,1)$
  & $(0,-1)$, $(-1,0)$ & $0$ & backward \\
$\mathcal{O}(1,1) \to \mathcal{O}$ & $(-1,-1)$ & $0$ & backward \\
\hline
\end{tabular}
\end{center}
Every backward or orthogonal twist has a factor $\mathcal{O}(-1)$ in
one tensor slot, hence vanishing $\RHom$; this proves
Lemma~\ref{lem:ext} in tabulated form. In particular
$B = \langle \mathcal{O}(1,0)\rangle \oplus
\langle\mathcal{O}(0,1)\rangle$ is a completely orthogonal pair.

\subsection{Gluing inequalities and the entry time}\label{app:entry}

Since the collection is strong, its shifts
$\mathcal{O}[n_0], \mathcal{O}(1,0)[n_1], \mathcal{O}(0,1)[n_1],
\mathcal{O}(1,1)[n_2]$ form an Ext-exceptional collection precisely
when the shift decreases by at least one across each forward
morphism (\S\ref{subsec:background}):
$n_0 - n_1 \geq 1$ and $n_1 - n_2 \geq 1$, the middle pair imposing
no condition by orthogonality and the skip pair
$\mathcal{O} \to \mathcal{O}(1,1)$ being implied. Normalizing the
heart to $\mathcal{P}((0,1])$, the simples sit at in-heart phases in
$(0,1]$, so the lifted phases satisfy
\[
  \gamma_1(t) = \varphi_t(B) - \varphi_t(\mathcal{O}) \in
  (n_0 - n_1 - 1,\; n_0 - n_1 + 1), \qquad
  \gamma_2(t) \in (n_1 - n_2 - 1,\; n_1 - n_2 + 1),
\]
and the Ext-exceptional (equivalently, Collins--Polishchuk)
condition is exactly $\gamma_1(t) > 1$ and $\gamma_2(t) > 1$. By
Appendix~\ref{app:saddle} both gaps are affine in $t$ with common
slope $s = 4m|\sin\theta|/\pi$ (Convention~\ref{conv:P1} orientation)
up to $O(1/t)$:
$\gamma_i(t) = s\,t + \gamma_i(0) + O(1/t)$, whence the entry time
\[
  t_0 = \frac{1}{s}\max\bigl(1 - \gamma_1(0),\, 1 -
  \gamma_2(0)\bigr)\,\bigl(1 + O(1/t_0)\bigr)
\]
of Proposition~\ref{prop:prod-entry}, the offsets $\gamma_i(0)$
being the prefactor arguments of Appendix~\ref{app:saddle} (fixed by
the $\widehat{\Gamma}$-integral normalization, which we leave
symbolic).

\subsection{The walls and their Jordan--H\"older data}
\label{app:walls}

\emph{Channel $\mathcal{O} \leftrightarrow B$.} A section of
$\mathcal{O}(1,0)$ vanishes on a divisor $D_1 \in |h_1|$, a vertical
line $\{p\}\times\PP^1$, giving
$0 \to \mathcal{O} \to \mathcal{O}(1,0) \to
\mathcal{O}(1,0)|_{D_1} \to 0$; since $h_1 \cdot h_1 = 0$ the
restriction has degree zero, so the cokernel is
$\mathcal{O}_{D_1}$, of class $[\mathcal{O}(1,0)] - [\mathcal{O}]$
and Chern character $h_1$. Rotating gives the triangle
$\mathcal{O}(1,0) \to \mathcal{O}_{D_1} \to \mathcal{O}[1]$. At the
wall $\gamma_1 = 1$ the phases of $\mathcal{O}(1,0)$ and
$\mathcal{O}[1]$ align, $Z(\mathcal{O}_{D_1}) = Z(\mathcal{O}(1,0))
+ Z(\mathcal{O}[1])$ with equal arguments, and $\mathcal{O}_{D_1}$
is strictly semistable with JH factors
$\{\mathcal{O}(1,0), \mathcal{O}[1]\}$. The horizontal line sheaves
$\mathcal{O}_{D_2}$, $D_2 \in |h_2|$, of class
$[\mathcal{O}(0,1)] - [\mathcal{O}]$, satisfy the mirror-image
statement; by the exact locking of Appendix~\ref{app:swap} their
phases coincide with those of the $\mathcal{O}_{D_1}$ \emph{for all}
$t$, so both families die at the same wall: multiplicity two.

\emph{Channel $B \leftrightarrow \mathcal{O}(1,1)$.} A section of
$\mathcal{O}(0,1)$ vanishes on $D' \in |h_2|$, giving
$0 \to \mathcal{O}(1,0) \to \mathcal{O}(1,1) \to
\mathcal{O}(1,1)|_{D'} \to 0$ with restriction of degree
$h_1 \cdot h_2 = 1$: the cokernel is $\mathcal{O}_{D'}(1)$, of class
$[\mathcal{O}(1,1)] - [\mathcal{O}(1,0)]$, dying at
$\gamma_2 = 1$ with JH factors
$\{\mathcal{O}(1,1), \mathcal{O}(1,0)[1]\}$; again doubled by the
locking.

\subsection{The deformation: exact splitting, two clocks, wall
splitting}\label{app:eps}

Set $q_1 = q e^{i\varepsilon}$, $q_2 = q$. The critical values
become $\lambda_{(s_1,s_2)} = 2m(s_1 e^{i\varepsilon/2} + s_2)$
exactly, so the middle splitting is
\[
  \Delta\lambda_B = \lambda(\mathcal{O}(0,1)) -
  \lambda(\mathcal{O}(1,0)) = 4m\bigl(e^{i\varepsilon/2} - 1\bigr)
  = 2im\varepsilon - \tfrac12 m\varepsilon^2 -
  \tfrac{i}{24} m\varepsilon^3 + O(\varepsilon^4),
\]
reproducing \eqref{eq:splitting}. Projecting onto the ray:
\begin{align*}
  \mathrm{Im}\bigl(\Delta\lambda_B\, e^{-i\theta}\bigr)
  &= 2m\varepsilon\cos\theta + \tfrac12 m\varepsilon^2\sin\theta
  + O(\varepsilon^3)
  && \Rightarrow \quad s_f = \tfrac{1}{\pi}\,
  \bigl|2m\varepsilon\cos\theta +
  \tfrac12 m\varepsilon^2\sin\theta\bigr|, \\
  \mathrm{Re}\bigl(\Delta\lambda_B\, e^{-i\theta}\bigr)
  &= 2m\varepsilon\sin\theta - \tfrac12 m\varepsilon^2\cos\theta
  + O(\varepsilon^3),
\end{align*}
the phase clock \eqref{eq:phaseclock} and (twice) the mass clock
\eqref{eq:massclock}: individually, the middle values
$\lambda \approx \pm i m\varepsilon$ give per-object mass rates
$\mathrm{Re}(\pm im\varepsilon\, e^{-i\theta}) = \pm
m\varepsilon\sin\theta$, so the two middle masses separate at
relative exponential rate $2m\varepsilon|\sin\theta|$ while each
drifts at $r_f = m\varepsilon|\sin\theta|$, the convention of
\eqref{eq:massclock}. The internal gap
$g_B(t) = s_f\, t + O(\varepsilon)$ reaches $1$ at
\[
  t_1 = \frac{\pi}{2m\varepsilon|\cos\theta|}
  \bigl(1 + O(\varepsilon)\bigr) \quad (\cos\theta \neq 0),
  \qquad
  t_1 = \frac{2\pi}{m\varepsilon^2|\sin\theta|}
  \bigl(1 + O(\varepsilon)\bigr) \quad (\cos\theta = 0),
\]
the second case because the leading imaginary part of
$\Delta\lambda_B$ drops out of the phase projection at
$\cos\theta = 0$ and the $O(\varepsilon^2)$ real part takes over ---
formula \eqref{eq:t1}. \emph{Wall splitting:} at the coarse wall
time $t_0 = O(1)$ the internal gap is already
$g_B(t_0) = s_f t_0 + O(\varepsilon) = O(\varepsilon) \neq 0$, so
the two channel-$1$ alignments (for $\mathcal{O}_{D_1}$ and
$\mathcal{O}_{D_2}$) occur at times differing by
$\Delta t = g_B(t_0)/s + O(\varepsilon^2) = O(\varepsilon)$, and
likewise in channel $2$: the two multiplicity-two walls of
Appendix~\ref{app:walls} unfold into four simple walls in two
$\varepsilon$-close pairs. \emph{Flicker times:} objects of mixed
middle class are direct sums $E_1 \oplus E_2$ with $E_i$ in the two
orthogonal summands; such a sum is semistable exactly when the
phases align modulo $2\ZZ$ of shifts, i.e.\ at the discrete times
$\{t : g_B(t) \in 2\ZZ\}$, spaced $2/s_f$ apart, and is never
stable --- the content of Proposition~\ref{prop:not-a-wall}.

\section{The determinant--period check at leading order}\label{app:det}

We verify the consistency check (i) of \S\ref{subsec:det}, step (2),
at leading order in $\hb$, on the symmetric slice of
\S\ref{subsec:A2bps}. Fix $b = 0$ and treat $a$ as the spectral
parameter of the lateral problem for
$\hb^2 y'' = (x^3 + ax)\,y$, with boundary conditions subdominant
along the positive real axis (the sector $S_0$, reached through the
barrier $(s,\infty)$) and in the sector $S_2$ reached across the
finite barrier (the conjugate choice $S_{-2}$ yields the conjugate
determinant): a resonance-type problem. By \cite{Sib} the
corresponding Stokes multiplier $D(a; \hb)$ is an entire function of
$a$; its zeros $a_n$ lie on the negative real axis to all orders in
$\hb$, with nonperturbatively small widths
$\Im a_n = O(e^{-\beta|a_n|^{5/4}/\hb})$
(see \S\ref{subsec:det-numerics}); it is the spectral determinant, and by the spectral
zeta theory of Voros \cite{Vor} it coincides with the
zeta-regularized determinant up to a factor
$\exp(\text{polynomial in } a)$ --- the same elementary-factor
ambiguity as in the $\tau = \tau_{\mathrm{PI}}$ identification of
\cite{BrMa}; this ambiguity is quantified in
\S\ref{subsec:det-elem}. Write
$\beta = B(\tfrac34, \tfrac32) \approx 0.958512$ for the period
constant of \eqref{eq:beta-period}, so that
$Z(\gamma_1)(a) = \beta|a|^{5/4}$ and
$Z(\gamma_2)(a) = i\beta|a|^{5/4}$ for $a < 0$.

\subsection*{Route $\alpha$: zero density to growth}
The eigenvalues quantize the oscillatory cycle:
$|Z(\gamma_2)(a_n)| = 2\pi\hb(n + \tfrac12) + O(\hb^2)$, \ie\
$\beta|a_n|^{5/4} = 2\pi\hb(n + \tfrac12)$ with $a_n$ on the
negative real axis, giving the counting function
$N(A) = \#\{|a_n| \leq A\} \sim c\,A^\rho$ with
$c = \beta/2\pi\hb$ and $\rho = 5/4$. (Cross-check of the order:
$\rho = \mathrm{wt}(Z)/\mathrm{wt}(a) = \tfrac{5/2}{2}$; in the
$b$-direction the same reasoning gives Sibuya's classical order
$\tfrac{5/2}{3} = \tfrac56$.) Since $\rho \in (1,2)$, $D$ is a
genus-one canonical product over the $a_n$, and the standard
Lindel\"of computation --- substitute $t = au$ in
$\int_0^\infty \log(1 + a/t)\, dN(t)$, integrate by parts, and use
the Mellin integral
$\int_0^\infty u^{\rho-1}(1+u)^{-1}du = \pi/\sin\pi\rho$, continued
to $\rho = \tfrac54$ through the genus-one subtraction --- yields,
for $a \to +\infty$ along the positive axis,
\begin{equation}\label{eq:app-lindelof}
  \log D(a) \;\sim\; \frac{\pi}{\sin \pi\rho}\; c\, a^{\rho}
  \;=\; -\pi\sqrt{2}\cdot\frac{\beta}{2\pi\hb}\, a^{5/4}
  \;=\; -\frac{\beta}{\sqrt{2}}\cdot\frac{a^{5/4}}{\hb}.
\end{equation}

\subsection*{Route $\beta$: analytic continuation of the periods}
The growth of a Stokes multiplier is the exponential of the Agmon
action between the sectors it connects: at leading order, one
\emph{half}-period per vanishing cycle adjacent to the boundary
sectors. Continuing from $a < 0$ to $a > 0$, the turning points
become $\{0, \pm i\sqrt{a}\}$ and the oscillatory cycle splits into
a complex-conjugate pair. Substituting $x = i\sqrt{a}\,u$,
\[
  x^3 + ax = i a^{3/2} u(1 - u^2), \qquad
  \sqrt{x^3 + ax}\; dx = e^{3i\pi/4} a^{5/4}
  \sqrt{u(1-u^2)}\; du,
\]
so the continued periods are
$Z_\pm = e^{\pm 3i\pi/4}\, \beta\, a^{5/4}$, and the real problem's
suppression sees both cycles symmetrically:
\begin{equation}\label{eq:app-agmon}
  \log D(a) \;\sim\; \frac{1}{2\hb}\bigl(Z_+ + Z_-\bigr)
  \;=\; \frac{\beta a^{5/4}}{\hb}\cos\tfrac{3\pi}{4}
  \;=\; -\frac{\beta}{\sqrt{2}}\cdot\frac{a^{5/4}}{\hb}.
\end{equation}

\subsection*{The agreement}
\eqref{eq:app-lindelof} and \eqref{eq:app-agmon} match exactly,
including the sign, from entirely different inputs: the
trigonometric factor $\pi/\sin\tfrac{5\pi}{4}$ of the zero-counting
integral \emph{is} the cosine produced by continuing the
Beta-period around the parameter space --- the counting integral
knows the monodromy of the elliptic curve. The leading-order
dictionary is therefore
\begin{equation}\label{eq:app-dict}
  \log D(a,b;\hb) \;=\; \frac{1}{2\hb}
  \sum_{\gamma \in V(D)} Z(\gamma)(a,b) \;+\; O(\log \hb^{-1}),
\end{equation}
where $V(D)$ is the set of vanishing cycles adjacent to the
determinant's boundary sectors, continued along the parameter ray;
monodromy of the parameters permutes $V(D)$ but preserves the sum,
consistently with single-valuedness. Differentiating
\eqref{eq:app-dict} under the period integral,
$\partial_b Z(\gamma) = \tfrac12 \oint_\gamma dx/\sqrt{Q_0}$, gives
the derivative-level statement
\[
  \partial_b \log D \;\sim\; \frac{1}{4\hb}
  \sum_{\gamma \in V(D)} \oint_\gamma \frac{dx}{\sqrt{Q_0}} :
\]
the growth of $\partial_b \log D$ is governed by the periods of the
\emph{first kind}. A third, independent data point: evaluating the
canonical product \emph{on} the zero axis (Levin's formula,
$\log|P(-r)| \sim \pi c\, r^\rho \cot\pi\rho$, and
$\cot\tfrac{5\pi}{4} = 1$) gives
$\log|D(a)| \sim \beta|a|^{5/4}/\hb = |Z(\gamma_1)|/\hb$ for
$a \to -\infty$: on that side the visible cycle set $V(D)$ is the
single real cycle $\gamma_1$ with weight one --- the full period,
as in \eqref{eq:det-singlecycle}; the corner height
$h(\pi) = \beta$ is confirmed by the wall map of
\S\ref{sec:outlook}. (The
growth-side constant, and with it the modulus of the cycle content
of $V(D)$
for $a \to +\infty$, is confirmed to seven digits in
\S\ref{subsec:det-numerics}.)

\subsection*{The symplectic constant}
The derivative statement above is the entry point to the family
geometry. Writing $A_i = \oint_{\gamma_i} dx/y$ and
$B_i = \oint_{\gamma_i} x\,dx/y$ on $E_{a,b}\colon y^2 = Q_0(x)$, so
that $dZ_i = \tfrac12(B_i\, da + A_i\, db)$, we claim
\begin{equation}\label{eq:app-symp}
  dZ(\gamma_1) \wedge dZ(\gamma_2)
  \;=\; \tfrac14\bigl(B_1 A_2 - A_1 B_2\bigr)\, da\wedge db
  \;=\; \pm\, 2\pi i \;\, da \wedge db ,
\end{equation}
the sign being the orientation of the symplectic basis
$\langle\gamma_1,\gamma_2\rangle = 1$: the period map
$(a,b) \mapsto (Z_1, Z_2)$ identifies the unfolding space, with its
constant symplectic form, with the central-charge space carrying
$dZ_1\wedge dZ_2$ --- the symplectic structure underlying the Joyce
formalism \cite{BrTau}, produced here from the determinant side.

\emph{Proof by residue.} By the Riemann bilinear relation for the
first-kind $\omega = dx/y$ against the second-kind
$\eta = x\,dx/y$ (double pole at $\infty$, no residue),
$A_1 B_2 - A_2 B_1 = 2\pi i \operatorname{Res}_\infty(F\eta)$ with
$dF = \omega$. In the local coordinate $t$ at infinity with
$x = t^{-2}$, $y = t^{-3}(1 + \tfrac{a}{2}t^4 +
\tfrac{b}{2}t^6 + \cdots)$:
\[
  \omega = -2\bigl(1 - \tfrac{a}{2}t^4 - \cdots\bigr)dt, \quad
  F = -2t + \tfrac{a}{5}t^5 + \cdots, \quad
  \eta = t^{-2}\omega = \bigl(-2t^{-2} + a t^2 + \cdots\bigr)dt,
\]
so the $t^{-1}$-coefficient of $F\eta$ is $(-2t)(-2t^{-2}) = 4/t$:
$\operatorname{Res} = 4$, giving $|B_1A_2 - A_1B_2| = 8\pi$ and
\eqref{eq:app-symp}. \emph{Check at the equianharmonic point}
($b = 0$): the periods evaluate to Beta functions,
$|A_i| = s^{-1/2}B(\tfrac14,\tfrac12)$,
$|B_i| = s^{1/2}B(\tfrac34,\tfrac12)$, and the two cross terms add
with equal phases, so the bilinear combination has modulus
$2\,B(\tfrac14,\tfrac12)B(\tfrac34,\tfrac12) =
2\,\Gamma(\tfrac14)\Gamma(\tfrac12)^2\Gamma(\tfrac14)^{-1}\cdot 4
= 8\pi$: the Beta identity
$B(\tfrac14,\tfrac12)B(\tfrac34,\tfrac12) = 4\pi$ \emph{is} the
Legendre relation at this point, and it reproduces the residue.

\emph{Flatness at leading order.} The Agmon connection form
$\theta = \tfrac{1}{2\hb}\sum_{V(D)} dZ(\gamma)$ appearing in
\eqref{eq:app-dict} is exact where $V(D)$ is locally constant, so
the curvature of the determinant line vanishes at order
$\hb^{-1}$; all of the order-$\hb^{-1}$ geometry resides in the
\emph{jumps} of $V(D)$ across the parameter Stokes rays, where
Picard--Lefschetz monodromy around $\Delta = 0$ acts by
$\gamma \mapsto \gamma \pm \langle\gamma,\gamma'\rangle\gamma'$ and
$\theta$ jumps by $\pm\tfrac{1}{2\hb}dZ(\gamma')$: a flat connection
with wall-crossing monodromy, measured by the symplectic pairing
\eqref{eq:app-symp} --- the leading-order shadow of the
Kontsevich--Soibelman structure, exactly as the Joyce-structure
picture requires.

\subsection*{The logarithmic coefficient vanishes}
The $O(\log\hb^{-1})$ term of \eqref{eq:app-dict} is absent for the
cubic: $\kappa_{\log} = 0$. (1) Since
$\deg Q_0 = 3$ is odd, the expansion
$\sqrt{Q_0} = x^{3/2} + \tfrac{a}{2}x^{-1/2} +
\tfrac{b}{2}x^{-3/2} - \tfrac{a^2}{8}x^{-5/2} + \cdots$ contains
only half-integer powers, never $x^{-1}$: the action primitive
$\tfrac25 x^{5/2} + a x^{1/2} - b x^{-1/2} + \cdots$ has no
logarithm, so Sibuya's normalization at infinity is log-free.
(2) The turning-point (Airy) factors carry the $\hb$-powers of the
connection formulas, and they cancel in the
$Q_0^{-1/4}$-normalized Wronskian ratios defining the multiplier.
(3) On the Riemann--Hilbert side, the prescribed asymptotics
$x_\gamma(\hb)e^{Z(\gamma)/\hb} \to \xi_\gamma$ of
\S\ref{subsec:BM} have \emph{constant} limits $\xi_\gamma$ --- no
logarithms --- so a nonzero $\kappa_{\log}$ would be inconsistent
with the solved RH problem. Consequently
$\log D = \tfrac{1}{2\hb}\sum_{V(D)} Z + O(1)$, the $O(1)$ term
being the logarithm of Sibuya's connection constant, and the
flatness statement above persists to order $\hb^0$. The contrast
with the Fano side is instructive: there the
$(\dim X/2 - m)\log w$ terms of \cite[(3.9)]{KRZ} originate in the
$(2\pi z)^{n/2}$ torus-volume prefactor of the oscillatory
integrals, a feature of the $n$-dimensional mirror absent from the
one-dimensional quantum mechanics.

\end{document}